\RequirePackage{fix-cm}
\RequirePackage{amsmath}
\documentclass{svjour3}                      
\smartqed  

\input{preamble}

\newcommand{\QPl}{\text{QPI}}             
\newcommand{\QPe}{\text{QPE}}             
\newcommand{\x}{{\mathbf x}}              
\newcommand{\s}{{\mathbf s}}              
\newcommand{\Ao}{{A}_0}         
\newcommand{\bo}{{b}_0}         
\newcommand{\co}{c_0}                     
\newcommand{\Fo}{\widetilde{F}_0}         
\newcommand{\fo}{\widetilde{f}_0}         
\newcommand{\dobj}{d_0}                   
\newcommand{\Fi}{\widetilde{F}_i}         
\newcommand{\ei}{e_i}                     
\newcommand{\Ai}{{A}_{i}}       
\newcommand{\bi}{{b}_{i}}       
\newcommand{\ci}{c_i}                     
\newcommand{\Aoii}{\widetilde{A}^{ii}_0}              
\newcommand{\Aoij}{\widetilde{A}^{ij}_0}              
\newcommand{\Ali}{\widetilde{A}_{i}}                  
\newcommand{\boi}{\widetilde{b}_{0i}}                 
\newcommand{\bli}{\widetilde{b}_i}                    
\newcommand{\cli}{c_i}                    
\newcommand{\Abco}{\overline{A}_0}        
\newcommand{\Abci}{\overline{A}_i}        
\newcommand{\blam}{\bm{\lambda}}          
\newcommand{\smsum}{{\textstyle \sum}}    
\newcommand{\ts}{t_{si}}                  
\newcommand{\tv}{t_{v}}                   
\newcommand{\N}{\mathcal{N}}                 
\newcommand{\g}{\mathbf{g}}               
\newcommand{\sproperty}{$\mathcal{S}$-property}     

\DeclareMathOperator{\diag}{diag}
\DeclareMathOperator{\trace}{tr}

\DeclareMathOperator{\evec}{\overline{E}}

\DeclareMathOperator{\id}{\rm{I}}
\DeclareMathOperator{\col}{col}
\DeclareMathOperator{\sign}{\rm{sign}}
\DeclareMathOperator{\range}{range}
\DeclareMathOperator{\nullspace}{null}
\DeclareMathOperator*{\argmin}{arg\,min}

\journalname{}
\begin{document}

\title{Strong duality in nonconvex quadratic problems with separable quadratic constraints
\thanks{This work was supported in part by the Spanish Ministry of 
Science and Innovation under the grant TEC2016-76038-C3-1-R (HERAKLES), 
the COMONSENS Network of Excellence TEC2015-69648-REDC and by an FPU
doctoral grant to Javier Zazo.}}

\titlerunning{Strong duality in nonconvex QPs with separable constraints}        

\author{Javier Zazo \and Santiago Zazo}

\institute{
           Information Processing and Telecommunications Center,\\
		   Universidad Polit\'ecnica de Madrid,\\
           ETSI Telecomunicaci\'on, Av. Complutense 30, 28040 Madrid (Spain).\\
           \email{javier.zazo.ruiz@upm.es and santiago.zazo@upm.es}
}

\date{Article draft from 2017.}

\maketitle

\begin{abstract}
We study nonconvex quadratic problems (QPs) with quadratic separable constraints, where
these constraints can be defined both as inequalities or equalities.
We derive sufficient conditions for these types of problems to present the \sproperty, which ultimately guarantees strong duality between the primal and dual problems of the QP.
We study the existence of solutions and propose a novel distributed algorithm to solve the problem optimally when the \sproperty\ is satisfied.
Finally, we illustrate our theoretical results proving that the robust least squares problem with multiple constraints presents the strong duality property.
\end{abstract}

\section{Introduction}
\label{sec:intro}
Nonconvex quadratically constrained quadratic problems (QCQPs) have numerous applications and have been extensively studied for more than 70 years~\cite{Feron2000}.
They are normally considered to be NP-hard problems, or at least, as hard as a large number of other problems which are also assumed to have a high degree of complexity~\cite{Garey1979}.
In this paper we consider a subclass of QCQPs where the constraints do not overlap with each other and we give sufficient conditions on the problem formulation that will guarantee strong duality between the primal problem and its dual.
From our results, we are able to identify specific cases in which these problems can be solved optimally, and provide some algorithmic results for these cases.
More formally, we specify the QCQP with separable constraints at the beginning of \cref{sec:analysis}.

The first result establishing strong duality for a nonconvex quadratic problem was established in~\cite{Yakubovich1971} with the proposal of the $\mathcal{S}$-Lemma.
Combined with previous results regarding the field of values of two matrices~\cite{Dines1941,Brickman1961,Hestenes1968}, strong duality was characterized for nonconvex quadratic problems with a single quadratic constraint.
An accesssible proof of such a result is available in~\cite[Appx. B]{Boyd2004} and an extension that considers a single equality constraint rather than an inequality is analyzed in~\cite{More1993}.
Strong duality for two quadratic constraints was analyzed in \cite{Polyak1998} for the real case and it was extended in \cite{Beck2006} to the complex case.
Extension results for the indefinite trust region subproblem are given in \cite{Fortin2004,Stern1994,Pong2014}.

However, results regarding an arbitrary number of constraints are very limited.
One relevant result is~\cite{Bengtsson2001}, which establishes strong duality in a beamforming application where the objective is convex and the constraints consist of differences of convex functions with a given structure.
In this case, the application is very specific, but the approach for the proof is novel.
Further results for similar beamforming settings have also been proposed in \cite{Huang2010}.

Another result that has an arbitrary number of constraints is established when the Hessian matrices corresponding to the objective and constraints are simultaneously diagonalizable via congruence~\cite{Jacobson1976}.
In this circumstance, the problem can be transformed into a diagonalized version and solved efficiently as a linear program.
However, such a requirement is extremely restrictive. 
Proposition 3.5 in \cite{Polik2007} presented another result on strong duality which requires that the Hessian matrices of the quadratic objective and constraints are a linear combination of only two of those matrices.
Other articles studying special structures can be found in \cite{Burer2013,Burer2015}.
Two survey papers 
can be found in~\cite{Hiriart-Urruty2002} and~\cite{Polik2007}.

Reference \cite{Zhang2000} studied QPs with convex separable constraints, and their results regarding strong duality requirements are similar to ours in certain cases.
However, we consider non-convex quadratic constraints, including both inequality and equality constraints, and our results are not easily obtained from their approach.
Article \cite{Jeyakumar2009} also studied strong duality of QCQPs with inequality constraints similar to \cite{Zhang2000} and ours, but do not consider equality constraints. 
To the best of our knowledge, separable equality constraints is novel in the literature.

Other results for nonconvex QCQPs are provided in~\cite{Jeyakumar2007}, which proposes necessary and/or sufficient conditions on some quadratic problems for optimality.
Specifically, they consider a weighted nonconvex least squares problem with convex quadratic constraints, and a partitioning problem (referred to as bivalent nonconvex problem), but do not establish strong duality requirements in these cases.
Similarly,~\cite{Beck2000} studies global optimality conditions for quadratic problems with binary constraints.
Both articles study the conditions to establish optimality for candidate points, but do not provide any methods nor guarantees to find such solutions.
Finally, \cite{Bomze2015} analyzes the use of copositive relaxations to obtain better lower bounds for general nonconvex QCQPs than semidefinite relaxations. 

Our contribution in this paper is twofold.
First, we propose sufficient conditions to establish strong duality for problems that have separable constraints, where the variables in the optimization problem appear in a single constraint.
Secondly, we focus on distributed algorithms that solve the QCQP, rather than a semidefinite program.
Our algorithms exploit the separability of the constraints and scales to high dimensions while converging to the global minimum.
To the best of our knowledge this is the first distributed algorithm with convergence guarantees to a global minimum in a non-convex setting.

At most, we can only have as many constraints in the problem as the dimension of the unknown variable.
Many applications fall into this category, namely, Boolean least squares, partitioning problems, the max-cut, robust least squares, localization problems, Euclidean matrix reconstruction, etc.
Since we consider both equality and inequality constraints, some of these instances can be analyzed within our framework to verify if they present the strong duality property.

The paper is structured as follows: at the end of this introduction we present some notation used throughout the paper. In \cref{sec:analysis} we introduce the QPs with inequality and equality constraints.
In \cref{sub:dual-formulation} we present the dual problem derived from the QCQP and define the constraint qualification known as the \sproperty\ necessary to establish strong duality between the primal and dual problems.
This results are well known in the literature.
In \cref{sec:main} we derive sufficient conditions for a QCQP with non-overlapping constraints to satisfy the \sproperty\, and in \cref{sec:existence} we derive sufficient conditions for the existence of solutions.
In \cref{sec:algorithms} we discuss different methods to find an optimal solution whenever the \sproperty\ is satisfied and, finally, in \cref{sec:rls} we apply our results to a robust least squares problem.
Some of the proofs have been moved to \cref{appendixA,appendixB,appendixC} to improve readability.

\subsection{Notation}
Throughout this paper we will use the following notation and operators:
The $p$-dimensional Euclidean space is denoted as $\mathbb{R}^p$, and $\mathbb{R}^p_+$ stands for the nonnegative orthant.
The set of symmetric matrices is denoted by $\mathcal{S}^p$, where these matrices have size $p\times p$.
For a matrix A, we write $A\succ 0$ to indicate that $A$ is positive definite, while $A \succeq 0$ means that $A$ is positive semidefinite.
We will use the symbol `$\unlhd$' to refer in a general way to either an equality or an inequality, when considering the constraints of an optimization problem.
More specifically, $\unlhd\in\set{\leq,\, =}$.
The trace of a square matrix $A$ is denoted by $\trace[A]$ and the \emph{Moore-Penrose inverse} by $A^{\dagger}$.
The block diagonal operator of matrices $A_1,\cdots,A_N$ is denoted by $\diag[A_1,\cdots,A_N]$.
The range and nullspace of a matrix A are denoted by $\range[A]$ and $\nullspace[A]$, respectively.
The identity matrix of dimension $p\times p$ is denoted by $\id_p$.
The element of a given matrix $A$ in row $i$ and column $j$ is denoted by $(A)_{ij}$.
We also use matrix $\evec$ of size $(p+1)\times (p+1)$ with single nonzero element $(\evec)_{p+1,p+1}=1$.

\section{Problem formulation}
\label{sec:analysis}
We consider the following QCQP with inequality constraints:
\begin{equation}
\begin{IEEEeqnarraybox}[][c]{r'l'l}
	\min_{\x} & \x ^T \Ao \x + 2 \bo ^T \x + \co \\
	\text{s.t.} & x_i^T \Ali x_i + 2 \bli ^T x_i + \cli \leq 0 & \forall i\in \N,
\end{IEEEeqnarraybox}
\tag{$\QPl$}
\label{QPl}
\end{equation}
where $\x \in \mathbb{R}^p$ is a column vector of size $p$;
$x_i \in \mathbb{R}^{n_i}$ is the $i$th subblock of vector $\x$ and has size $n_i$;
we suppose that $\x$ has $N$ blocks, and writing $\N=\set{1,\ldots,N}$, we have $\x=(x_i)_{i\in\N}$ and $p=\sum_{i\in\N}n_i$.
Matrix $\Ao\in\mathcal{S}^p$ is a symmetric matrix of size $p\times p$;
$\bo\in\mathbb{R}^p$ is a column vector;
$\cli$ for all $i$ and $\co\in\mathbb{R}$ are scalar numbers;
$\Ali\in\mathcal{S}^{n_i}$ for all $i$ are symmetric matrices; and finally,
$\bli\in\mathbb{R}^{n_i}$ for all $i$ are column vectors.
Note that we did not assume that $\Ao$ and $\Ali$ are positive semidefinite matrices, and, therefore, the problem is in general nonconvex.

In addition to problem \ref{QPl}, we also consider the following QCQP:
\begin{equation}
\begin{IEEEeqnarraybox}[][c]{r'l'l}
	\min_{\x} & \x ^T \Ao \x + 2 \bo ^T \x + \co \\
	\text{s.t.} & x_i^T \Ali x_i + 2 \bli ^T x_i + \cli = 0 & \forall i\in \N,
\end{IEEEeqnarraybox}
\tag{$\QPe$}
\label{QPe}
\end{equation}
where we substituted the inequality constraints with equality constraints.
In general our results generalize to QCQPs which have both separable equality and inequality constraints, but we omit this description as it is straightforward to deduce from our results.

We introduce some notation that will be useful for our presentation
\begin{IEEEeqnarray}{rCl}
	\bi & = & [0^T_{n_1\times 1},\ldots,\bli^T,\ldots,0^T_{n_N\times 1}]^T,   \\
	\Ai & = & \diag [0_{n_1\times n_1},\ldots,\Ali,\ldots,0_{n_N\times n_N}],   \label{eq:constraint-hessian}
\end{IEEEeqnarray}
where the previous expressions have nonzero elements in $i$th positions, $\bi\in\mathbb{R}^p$ and $\Ai\in\mathcal{S}^{p}$.
Furthermore, we also introduce extended matrices
\begin{IEEEeqnarray}{l"l"l}
	\Abco = 
	\begin{pmatrix}
		\Ao & \bo \\
		\bo^T & \co
	\end{pmatrix}, &
	\Abci = 
	\begin{pmatrix}
		\Ai & \bi \\
		\bi^T & \cli
	\end{pmatrix} & 
	\forall i\in\N
\label{eq:Abco-Abci}
\end{IEEEeqnarray}
that combine all the elements that appear in the quadratic problem.
Note that $\Abco$ and $\Abci\in \mathcal{S}^{p+1}$.

We define the following functions: 
\begin{equation}
	f(\x)=\x ^T \Ao \x + 2 \bo ^T \x + \co,
\end{equation}
which is the objective function of the QP; and, with a slight abuse of notation,
\begin{equation}
	g_i(x_i)  =  x_i^T \Ali x_i + 2 \bli ^T x_i + \cli
              =  \x^T \Ai \x + 2 \bi ^T \x + \cli = g_i(\x), 
\end{equation} 
which denotes the $i$th constraint.

We make the following assumption: 
\begin{assumption}[Slater's condition]
Assume the following:
\begin{enumerate}[label=\Alph*:]
 	\item  For problem \ref{QPl},  
        $\exists \hat{x}_i\in\mathbb{R}^{n_i}\,|\,g_i(\hat{x}_i)<0$ for all $i\in\N$.
 	\item  For problem \ref{QPe},
        $\exists \hat{x}^1_i,\hat{x}^2_i \in\mathbb{R}^{n_i}\,|\,g_i(\hat{x}^1_i) < 0 \wedge  g_i(\hat{x}^2_i)>0$ for all $i\in\N$.
        \label{assumption:slater-B}
\end{enumerate}
\label{assumption:slater}
\end{assumption}
The previous assumption guarantees some regularity condition that will be necessary in the upcoming derivations.
We remark that \ref{assumption:slater-B} was introduced as a constraint qualification for a QP with a single quadratic equality constraint in \cite{More1993}.

\subsection{Dual formulation}
\label{sub:dual-formulation}
QCQPs have been studied in \cite{Nesterov2000}, including the duals.
The Lagrangian of the QCQPs has the following form
\begin{equation}
	L(\x,\blam)= \x^T \Ao \x + 2\bo^T \x +\co +\sum_{i\in\N} \lambda_i \big( \x^T \Ai \x + 2 \bi^T \x + \cli \big),
\label{eq:lagrangian-QP}
\end{equation}
where we introduced dual variables $\blam=(\lambda_i)_{i\in\N}$ and $\lambda_i\in\mathbb{R}$.
The dual function $q(\blam) = \min_{\x} L(\x,\blam)$ is given by
\begin{equation}
	q(\blam) 	=  
	\begin{cases}	
			-(\bo+\sum_i \lambda_i \bi^T) (\Ao+\sum_i \lambda_i \Ai)^\dagger (\bo+\sum_i \lambda_i \bi^T) 
	  		  +\co+ \sum_{i}\lambda_i \cli \\
	  		 \hspace{3em} \text{if } \Ao+\sum_i \lambda_i \Ai \succeq 0 
	  		 \text{ and }(\bo+\sum_i \lambda_i \bi)\in\range(\Ao+\sum_i \lambda_i \Ai) \\
	  		 -\infty \hfill \text{otherwise.}
	\end{cases}
\label{eq:dual-function-QP}
\end{equation}
The dual problems of \ref{QPl} and \ref{QPe} can be rewritten using the Schur complement:
\begin{equation}
\begin{IEEEeqnarraybox}[][c]{r'l}
	\max_{\gamma,\blam} & \gamma \\
	\text{s.t.} &
	\begin{pmatrix}
		\Ao+\sum_i \lambda_i \Ai & \bo+\sum_i \lambda_i \bi \\
		(\bo+\sum_i \lambda_i \bi)^T & \co+ \sum_{i}\lambda_i \cli-\gamma
	\end{pmatrix} \succeq 0 \\
	& \gamma\in\mathbb{R},\quad\blam \in \Gamma
\end{IEEEeqnarraybox}
\label{eq:dual-problem-QPl}
\end{equation}
where we introduced the set $\Gamma$ as a general way to constrain $\blam$ in problem \ref{QPl} and \ref{QPe} accordingly, i.e., $\Gamma = \mathbb{R}^N_+$ in \ref{QPl}, and $\Gamma = \mathbb{R}^N$ in \ref{QPe}.
Furthermore, $\Gamma=\Gamma_1 \times \cdots \times \Gamma_N$, where $\Gamma_i = \mathbb{R}_+$ in \ref{QPl} or $\Gamma_i = \mathbb{R}$ in \ref{QPe}.

The dual problem given by \eqref{eq:dual-problem-QPl} is an underestimator of the optimal values of \ref{QPl} and \ref{QPe}.
Under the strong duality property the solution of the primal problem and the dual problem coincide, and  can be efficiently determined.

In order to derive the conditions that guarantee strong duality, we will be using results from~\cite{Jeyakumar2007}.
Before introducing the theorem that will motivate our proof, we first define the constraint qualification required for our QCQPs:
\begin{definition}[\sproperty\ \cite{Jeyakumar2007}]
A QCQP satisfies the \sproperty, if and only if, the following system of equations are equivalent $\forall\alpha\in\mathbb{R}$:
\begin{IEEEeqnarray}{l}
	g_i(x_i) \unlhd 0,\;\forall i\in\N \Longrightarrow f(\x)\geq \alpha, \quad \Longleftrightarrow
	\IEEEyesnumber \IEEEyessubnumber \label{def:sproperty-qpl-1} \\
	\exists \blam\in\Gamma\,|\,f(\x)+\sum_{i\in\N}\lambda_i g_i(x_i)\geq \alpha,\:\forall \x\in\mathbb{R}^p.
	\IEEEyessubnumber \label{def:sproperty-qpl-2}
\end{IEEEeqnarray}
\label{def:sproperty-qpl}
\end{definition}
We use the symbol `$\unlhd$' to refer in a general way to either an equality or an inequality, when considering the constraints of an optimization problem.
More specifically, $\unlhd\in\set{\leq,\, =}$.

It is immediate to see that \eqref{def:sproperty-qpl-2} $\Rightarrow$ \eqref{def:sproperty-qpl-1}: suppose \eqref{def:sproperty-qpl-2} is satisfied and $g_i(x_i)\unlhd 0$.
Then $\sum_{i\in\N}\lambda_i g_i(x_i) \leq 0$ and we necessarily have that $f(\x)\geq \alpha$.
The other implication is not always true and is difficult to know when it is satisfied.
The main result of this paper is dedicated to establish sufficient requirements that make such implication  hold.

We can express the system of equations from the definition of the \sproperty\  as a system of alternatives rather than a system of equivalencies.
This is helpful because we will be able to derive a theorem of alternative systems based on the new formulation.
A description of the theory of alternatives can be found in section 5.8 in \cite{Boyd2004}.
We reformulate the \sproperty\ with the following lemma:
\begin{lemma}[\sproperty\  reformulation]
	A QCQP possesses the \sproperty\,  if and only if, the following systems are strong alternatives for every $\alpha\in\mathbb{R}$:
	\begin{IEEEeqnarray}{l} \IEEEyesnumber
		\exists \x\in\mathbb{R}^p\,|\,f(\x) < \alpha\,\wedge\,g_i(x_i) \unlhd 0\quad\forall i\in\N, \quad\Longleftrightarrow  
		\IEEEyessubnumber \label{subeq:rank1-sproperty-1} \\
		\nexists \blam\in\Gamma\,|\,\Abco+\sum_{i\in\N} \lambda_i \Abci -\alpha \evec \succeq 0.
		\IEEEyessubnumber \label{eq:sa-21}
	\end{IEEEeqnarray}
\label{lemma:strong-alternatives}
\end{lemma}
We introduced $\evec$ to denote a square matrix of size $(p+1)$ with single nonzero element $(\evec)_{p+1,p+1}=1$.
\begin{proof}
See \cref{appendixA}.
\end{proof}

The following theorem reflects the importance of the \sproperty:

\begin{theorem}[Necessary and sufficient conditions, Theorem 3.1 in~\cite{Jeyakumar2007}]
	Let $\x^{\ast}$ be a feasible point of the QCQP.
	Suppose that the problem satisfies the \sproperty\ and $\blam\in \Gamma$.
	Then, $\x^{\ast}$ is a global minimizer of the QCQP if and only if 
	\begin{IEEEeqnarray}{l} \IEEEyesnumber
		\nabla_{\x} f(\x^{\ast})+\sum_{i\in\N} \lambda_i^{\ast} \nabla_{\x} g_i(\x^{\ast}) = 0, \IEEEyessubnumber \\
		\Ao + \sum_{i\in\N} \lambda_i^{\ast} \Ai \succeq 0, \quad   \lambda_i^{\ast} g_i(\x^{\ast}) = 0, \quad \forall i\in\N. \IEEEyessubnumber 
	\end{IEEEeqnarray}
\label{th:jeyakumar}
\end{theorem}
\Cref{th:jeyakumar} states that if the \sproperty\  holds in the QCQP, then, the KKT conditions of the problem plus $\Ao + \sum_{i\in\N} \lambda_i^{\ast} \Ai \succeq 0$ are sufficient and necessary to establish optimality of the solution.
The next proposition explicitly relates the \sproperty\ with zero gap between the primal and dual problems:
\begin{proposition}
	Strong duality holds between \ref{QPl}, \ref{QPe} and problem \eqref{eq:dual-problem-QPl}, if and only if the \sproperty\ is satisfied.
\label{proposition:strong-duality}
\end{proposition}
Sufficient conditions to satisfy the \sproperty\ will be presented in \cref{theorem:sproperty-result} and \cref{theorem:sproperty-result2}.
We remark that the \sproperty\ is a constraint qualification. 
It has the interesting consequence that it only occurs whenever there is zero duality gap between the primal and dual problems of the QCQP.
Therefore, there is an immediate equivalence between strong duality and the \sproperty.
\begin{proof}
Suppose strong duality holds. 
Then, the KKT conditions derived from \eqref{eq:lagrangian-QP} and $\Ao + \sum_{i\in\N} \lambda_i^{\ast} \Ai \succeq 0$ from problem \eqref{eq:dual-problem-QPl} are satisfied for optimal points $\x^{\ast}$ and $\blam^{\ast}$.
Because of \cref{th:jeyakumar} the \sproperty\ holds.

Suppose now that the \sproperty\ holds.
Assume that the QCQP has a minimum value which we denote with $\gamma$:
\begin{equation}
	\forall \x\in\mathbb{R}^p\,|\,\x^T \Ai \x + 2 \bi^T \x + \cli \unlhd 0\quad (\forall i\in\N) \Longrightarrow 
	\x^T \Ao \x + 2 \bo^T \x + \co \geq \gamma
\end{equation}
Because of the \sproperty, there exists a $\blam\in\Gamma$ such that $\Abco-\gamma \evec+\sum_i\Abci\succeq 0$ which makes problem \eqref{eq:dual-problem-QPl} feasible.
Because of this, strong duality holds.
\end{proof}

\section{Analysis of the \texorpdfstring{\sproperty}{S-property} in QCQPs}
\label{sec:main}
\Cref{lemma:strong-alternatives} establishes necessary and sufficient conditions for the \sproperty\ to hold.
In this section we will derive sufficient conditions that will guarantee that these requirements are satisfied.

The road map is as follows: we first analyze a system of strong alternatives in the context of semidefinite programming (SDP), whose variable $Z$ can be a matrix of any rank.
Then, we reformulate the system to an equivalent one with easier to analyze properties, and derive the conditions that guarantee the existence of a rank 1 solution in such system.
Finally, we will relate the existence of a rank 1 solution with \cref{lemma:strong-alternatives} and state our main result over the \sproperty.
The importance of establishing a rank 1 solution resides in being able to recover the QCQP from the SDP formulation.
These steps are summarized in \cref{fig:block-diagram}.
Note that the equivalences may hold under specific requirements, which we will state accordingly.
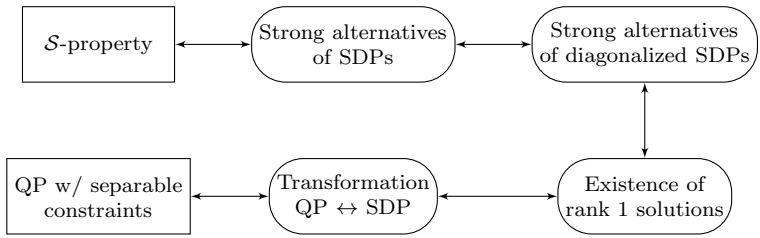
\begin{figure}
	\centering
	\begin{tikzpicture}[>=latex']
        \tikzset{block/.style= {draw, rectangle, align=center,minimum width=2cm,minimum height=1cm},
        rblock/.style={draw, shape=rectangle,rounded corners=1.5em,align=center,minimum width=2cm,minimum height=1cm}
        }
        \node [block]  (sproperty) {\sproperty};
        \node [rblock, right = of sproperty] (sdp1) {Strong alternatives\\ of SDPs};
        \node [rblock, right = of sdp1] (sdp2) {Strong alternatives \\ of diagonalized SDPs};
        \node [block, below = of sproperty] (qp) {QP w/ separable\\ constraints};
        \node [rblock, below = of sdp1] (qp_transf) {Transformation\\ QP $\leftrightarrow$ SDP};
        \node [rblock, below = of sdp2] (rank1) {Existence of\\ rank 1 solutions};

        \path[draw,<->] (sproperty) edge (sdp1)
                    (sdp1) edge (sdp2)
                    (sdp2) edge (rank1)
                    (rank1) edge (qp_transf)
                    (qp_transf) edge (qp)
                    ;
\end{tikzpicture}
	\caption{Block diagram of equivalences to prove that the \sproperty\ holds.}
	\label{fig:block-diagram}
\end{figure}

The next lemma arises from a known result in the literature, which establishes that systems of linear matrix inequalities (LMIs) satisfy the \sproperty. It is expressed in terms of strong alternatives.
\begin{lemma}
Under \cref{assumption:slater}, the following systems are strong alternatives for all $\alpha\in\mathbb{R}$:
\begin{IEEEeqnarray}{l} \IEEEyesnumber\phantomsection \label{eq:sdp-strong-alternative-systems}
    \exists Z \succeq 0\,|\,\trace[(\Abco-\alpha \evec)Z]<0\;\wedge\;\trace[\Abci Z] \unlhd 0\quad\forall i\in\N, \quad \Longleftrightarrow 
    \IEEEyessubnumber \label{subeq:sdp-strong-alternatives1}\\
	\nexists \blam\in\Gamma\,|\,\Abco+\sum_{i\in\N}\lambda_i \Abci -\alpha \evec \succeq 0. \IEEEyessubnumber
\end{IEEEeqnarray}
\label{lemma:sdp-strong-alternatives}
\end{lemma}
\begin{proof}
See \cref{appendixA}.
\end{proof}

We will perform a transformation of the system in \cref{subeq:sdp-strong-alternatives1} so that the matrices in the constraints become diagonal or \emph{almost} diagonal (some elements may remain in the last column/row as we describe later).
These reformulations will be helpful because the new forms will be easier to study.
Before analyzing them, we introduce some definitions regarding the theory of matrix congruency:

\begin{definition}
Two matrices $A$ and $B$ over a field are called congruent if there exists an invertible matrix $P$ over the same field such that
\begin{equation}
	P^T A P = B.
\end{equation}
\end{definition}

\begin{definition} \label{def:diagonalizable}
A set of matrices $\set{A_1,A_2,\ldots,A_N}$ is said to be simultaneously diagonalizable via congruence, if there exists a nonsingular matrix $P$ such that $P^T \Ai P$ is diagonal for every matrix $\Ai$.
\end{definition}

In general, finding sets of matrices that are simultaneously diagonalizable is a difficult task and it remains an open question \cite{Hiriart-Urruty2002}.
Nonetheless, we will study one simple case of simultaneously diagonalizable matrices.
Finally, the following theorem emphasizes an important property of congruent systems:

\begin{theorem}[Sylvester's law of inertia, theorem 4.5.8 in \cite{Horn2012}]:
Two congruent matrices $A$ and $B$ 
have the same number of positive and negative eigenvalues.
This implies that the definiteness of matrix $A$ is preserved onto $B$.
\label{remark:sylvester}
\end{theorem}

Consider the following matrix $P$ of dimension $(p+1)\times (p+1)$
\begin{equation}
	P=
	\begin{pmatrix}
		P_1 & 0_{n_1 \times n_2} & \cdots & p_{1}\\
		0_{n_2 \times n_1} & P_2 & \cdots & p_{2}\\
		\vdots & \ddots & \ddots & \vdots \\
		0_{n_N \times n_1} & \cdots & P_N & p_{N}\\
		0_{1 \times n_1} & 0_{1 \times n_2} & \cdots & 1
	\end{pmatrix},
\label{eq:p-matrix}
\end{equation}
where $P_i$ has dimension $n_i\times n_i$ and $p_i\in\mathbb{R}^{n_i}$.
We obtain the following congruent matrices for every $i\in\N$:
\begin{equation}
	F_i = P^T \Abci P=
	\begin{pmatrix}
	    0 & 0 & \cdots &  0 \\
		\vdots & P_i^T \Ali P_i & 0 & P_i^T(\Ali p_{i} + \bli)\\
		\vdots & \cdots & \ddots & 0\\
		0 & (p_{i}^T \Ali^T +\bli^T)P_i & 0 & p_{i}^T \Ali p_i+2\bli^T p_{i}+\cli
	\end{pmatrix}.
\end{equation}

We consider two cases.
The first case assumes that $\bli\in\range[\Ali]$ for every $i\in\N$.
Then, since every $\Ali$ are real symmetric matrices, they allow a spectral decomposition of the form $\Ali=Q_i\Delta_i Q^T_i$ where $Q_i$ are orthogonal matrices, i.e., $Q_i Q_i^T=Q_i^T Q_i=\id_{n_i}$, and $\Delta_i$ are diagonal.
Choosing values $P_i = Q_i \Delta_i^{\dagger/2}$
and $p_i=-\Ali^{\dagger} \bli$ satisfy that every $F_i=P^T \Abci P$ for all $i\in\N$ are simultaneously diagonal.
Moreover, this choice of values ensures that $P$ is invertible.
Finally, we note that $F_0=P^T(\Abco-\alpha \evec) P$ is not diagonal in general.

The second case assumes that for some $i\in\N$, $\bli\notin\range[\Ali]$.
In this case, $F_i$ will not become diagonal since it is not possible that $P_i^T(\Ali p_{i} + \bli) = 0$ for any $P_i$ which is also full rank.
Nonetheless, the previous choice of matrix $P$ ensures that some components will be \emph{almost} diagonalized, and some elements will remain in the last column and row.
However, if for some $j$ index $(F_i)_{p+1,j} = (F_i)_{j,p+1} \neq 0$, we necessarily have that $(F_i)_{jj} = 0$.
This structure of $F_i$ can still be exploited to obtain strong duality results.

Now we can reformulate \cref{subeq:sdp-strong-alternatives1} to an equivalent system as follows:
\begin{lemma}
The following systems are equivalent:
\begin{IEEEeqnarray}{l} \IEEEyesnumber
	\exists Z\succeq 0 \,|\,\trace[(\Abco-\alpha\evec) Z]<0\,\wedge\,\trace[\Abci Z] \unlhd 0 \quad \forall i\in\N,\quad \Longleftrightarrow 
	\IEEEyessubnumber \label{eq:ZY-domain-1} \\
	\exists Y\succeq 0 \,|\,\trace[F_0 Y]<0\,\wedge\,\trace[F_i Y] \unlhd 0 \quad \forall i\in\N. 
	\IEEEyessubnumber \label{eq:ZY-domain-2}
\end{IEEEeqnarray}
\label{lemma:equivalence-between-Z-and-Y}
\end{lemma}
\begin{proof}
We consider the change of variable $Y=P^{-1}ZP^{-T}$ and assume there exists some $Z\succeq 0$ satisfying \eqref{eq:ZY-domain-1}.
Then,
\begin{IEEEeqnarray*}{l}
    \trace[(\Abco-\alpha\evec) Z] = \trace[P^{-T}F_0 P^{-1}Z]=\trace[F_0 P^{-1}ZP^{-T}]=\trace[F_0 Y]<0 \\
	\trace[\Abci Z] = \trace[P^{-T}F_i P^{-1}Z]=\trace[F_i P^{-1}ZP^{-T}]=\trace[F_i Y] \unlhd 0,\quad \forall i\in\N.
\end{IEEEeqnarray*}
Because of Sylvester's law of inertia, 
$Y$ is a congruent matrix of $Z$ and we have that $Y\succeq 0$.
This implies that \eqref{eq:ZY-domain-1} and \eqref{eq:ZY-domain-2} are equivalent.
\end{proof}

The next step is to find when there exists a $Y\succeq 0$ that satisfies \eqref{eq:ZY-domain-2} and has rank~1.
This is the most important step to prove strong duality in a QCQP, and will introduce several limitations to when this property can be guaranteed.
We analyze the two cases separately, explained in \cref{sub:Fi-diagonalizable,sub:Fi-not-diagonalizable}, respectively.

\subsection{\texorpdfstring{F\textsubscript{i}}{Fi} simultaneously diagonalizable}
\label{sub:Fi-diagonalizable}
This subsection assumes that $\bli\in\range[\Ali]$ and that we can find $F_i$ diagonal matrix and congruent to $\Abci$ for all $i\in\N$ simulatenously.
We state the following:

\begin{theorem}
Assume $\exists Y\succeq 0$ such that $Y$ satisfies $\trace[F_0 Y]<0$ and a system of constraints given by $\trace[F_i Y]=\psi_i$, where $\psi_i\in\mathbb{R}$ for all $i\in\N$.
A sufficient condition for the previous system to have a rank 1 solution is that there exists a diagonal matrix $D$ with entries $\pm 1$ such that $D F_0 D$  is a $Z$-matrix.
\label{theorem:rank1}
\end{theorem}
We recall that a $Z$-matrix is a symmetric matrix whose off-diagonal elements are non-positive.
Note also that diagonal matrices are by definition $Z$-matrices.
\begin{proof}
See \cref{appendixA}.
\end{proof}

\Cref{theorem:rank1} requires there exists a matrix $D$ such that $D F_0 D$ is a $Z$--matrix to guarantee existence of a rank 1 solution.
Accordingly, this implies that there exists a vector $y=(y_j)_{j=1}^{p+1}$ and $y_j=(D)_{jj} \sqrt{\sum_{r=1}^R v^2_{rj}}$ such that
\begin{IEEEeqnarray}{l} \IEEEyesnumber\phantomsection \label{eq:existence-sol-rank1}
	y^T F_0 y \leq \trace[F_0 Y] < 0 \IEEEyessubnumber \\
	y^T F_i y = \trace[F_i Y] = \psi_i,\quad \forall i\in\N. \IEEEyessubnumber
\end{IEEEeqnarray}
The rank 1 solution in \eqref{eq:existence-sol-rank1} implies that the system of strong alternatives from \cref{lemma:sdp-strong-alternatives} is satisfied because of the relation established in \cref{lemma:equivalence-between-Z-and-Y}.
We can now combine these results with the \sproperty\ and show that \eqref{eq:existence-sol-rank1} is equivalent to \cref{subeq:rank1-sproperty-1}.
We complete the argument with the following two lemmas: 
\begin{lemma}
The following systems are equivalent:
\begin{IEEEeqnarray}{l} \IEEEyesnumber \phantomsection \label{eq:vector-to-equation}
	\exists y\in\mathbb{R}^{p+1}\,|\, y^T F_0 y < 0\,\wedge\,y^T F_i y\leq 0\quad\forall i\in\N,\quad \Longleftrightarrow 
	\IEEEyessubnumber \label{eq:vector-to-equation-1} \\
	\exists \x\in\mathbb{R}^p\,|\,f(\x) < \alpha\,\wedge\,g_i(\x)\leq 0,\,\forall i\in\N . \IEEEyessubnumber \label{eq:vector-to-equation-2}
\end{IEEEeqnarray}
\label{lemma:from-sdp-to-QPl}
\end{lemma}
\begin{proof}
See \cref{appendixA}.
\end{proof}

Problem \ref{QPe} requires an alternative lemma:
\begin{lemma}
The following systems are equivalent if $\Ali\neq 0\quad \forall i\in\N$:
\begin{IEEEeqnarray}{l} \IEEEyesnumber
	\exists y\in\mathbb{R}^{p+1}\,|\, y^T F_0 y < 0\,\wedge\,y^T F_i y = 0\quad\forall i\in\N, \quad \Longleftrightarrow 
	\IEEEyessubnumber \label{eq:vector-to-equation-3}\\
	\exists \x\in\mathbb{R}^p\,|\,f(\x) < \alpha\,\wedge\,g_i(\x) = 0,\,\forall i\in\N. 
	\IEEEyessubnumber \label{eq:vector-to-equation-4}
\end{IEEEeqnarray}
\label{lemma:from-sdp-to-QPe}%
\end{lemma}
\begin{proof}
See \cref{appendixA}.
\end{proof}

We now state our main result.
\begin{theorem}
Given a QCQP in the form of \ref{QPl} or \ref{QPe}, suppose \cref{assumption:slater} is satisfied and that $\bli\in\range[\Ali]$ for every $i\in\N$.
Furthermore, assume that there exists a diagonal matrix $D$ whose elements are $\pm 1$ such that $D F_0 D$ is a $Z$--matrix.
Then, the \sproperty\ holds.
\label{theorem:sproperty-result}
\end{theorem}
An immediate consequence of \cref{theorem:sproperty-result} is that we now have sufficient conditions to apply \cref{th:jeyakumar} and \cref{proposition:strong-duality}.

\begin{proof}
We follow the steps depicted in \cref{fig:block-diagram}.
\Cref{lemma:sdp-strong-alternatives} introduces a system of strong alternatives between semidefinite programs.
\Cref{lemma:equivalence-between-Z-and-Y} reformulates one of the systems as an equivalent problem with diagonal matrices.
\Cref{theorem:rank1} introduces a sufficient condition for the reformulated system to have a solution of rank 1.
\Cref{lemma:from-sdp-to-QPl,lemma:from-sdp-to-QPe} show that if there exists a rank 1 solution in the system of strong alternatives, then the \sproperty\ defined in \cref{lemma:strong-alternatives} is satisfied.
\end{proof}

\begin{remark}
We note that if $D F_0 D$ becomes diagonal, then the QCQP can be formulated as an equivalent linear program, as presented in \cite{Jacobson1976}, and strong duality holds.
\end{remark}

\subsection{\texorpdfstring{F\textsubscript{i}}{Fi} not simultaneously diagonalizable}
\label{sub:Fi-not-diagonalizable}
In this subsection we assume that for some $i\in\N$, $\bli\notin\range[\Ali]$.
Denote with $M$ the set of indexes such that for $m\in M$ we have $(F_i)_{p+1,m} = (F_i)_{m,p+1} \neq 0$ for any $i\in\N$.
Furthermore, we necessarily have that $(F_i)_{mm} = 0$ for $m\in M$.
We can invoke the following theorem for this case:

\begin{theorem}
Suppose $\exists Y\succeq 0$ such that $Y$ satisfies $\trace[F_0 Y]<0$ and $\trace[F_i Y]=\psi_i$, where $\psi_i\in\mathbb{R}$ for all $i\in\N$.
A sufficient condition for the previous system to have a rank~1 solution is that the diagonal elements of $\Ao$ are nonnegative, $(\Ao)_{mj}=(\Ao)_{jm}=0$ for any $m\neq j$ ($m\in M$), and there exists a diagonal matrix $D$ whose terms satisfy
	\begin{equation}
		(D)_{ii} (D)_{jj} (F_0)_{ij} \leq 0\qquad \forall i,j\notin M.
	\end{equation}
\label{th:2ndcase}
\end{theorem}

\Cref{th:2ndcase} is more demanding than \cref{theorem:rank1} because of the lack of simultaneous diagonalization of the constraints.
Nevertheless, both theorems state unique results regarding the existence of rank~1 solutions in SDP programs.
The condition that $(\Ao)_{mj}=(\Ao)_{jm}=0$ for any $m\neq j$ ($m\in M$) means that the column/row associated with element $(\Ao)_{mm}$ are zero except for term $(\Ao)_{mm}$.

\begin{proof}
See \cref{appendixA}.
\end{proof}

The final step is to relate the existence of $y\in\mathbb{R}^{p+1}$ such that 
\begin{IEEEeqnarray}{rCl'l}
	y^T F_0 y & \leq & \trace[F_0 Y] < 0 \\
	y^T F_i y & = & \trace[F_i Y] \unlhd 0 & \forall i\in\N,
\end{IEEEeqnarray}
to the \sproperty\ reformulation described in \cref{lemma:strong-alternatives}.
We can state the following:
\begin{lemma}
	Assume that $\Ao\succeq 0$. Then, the following systems are equivalent:
	\begin{IEEEeqnarray}{l} \IEEEyesnumber
		\exists y\in\mathbb{R}^{p+1}\,|\, y^T F_0 y < 0\,\wedge\,y^T F_i y \unlhd 0\quad\forall i\in\N, \quad \Longleftrightarrow 
		\IEEEyessubnumber \label{eq:vector-to-equation-5}\\
		\exists \x\in\mathbb{R}^p\,|\,f(\x) < \alpha\,\wedge\,g_i(\x) \unlhd 0,\,\forall i\in\N. 
		\IEEEyessubnumber \label{eq:vector-to-equation-6}
	\end{IEEEeqnarray}%
\label{lemma:case2-back-to-sproperty}%
\end{lemma}%
\vspace{-3em}
\begin{proof}%
See \cref{appendixA}.
\end{proof}
The previous lemma directly relates the existence of rank~1 solution in the system of LMIs with the \sproperty.
This allows us to summarize this result in the following theorem.

\begin{theorem}
	Given a QCQP in the form of \ref{QPl} or \ref{QPe}, suppose \cref{assumption:slater} is satisfied and that for some $i\in\N$, $\bli\notin\range[\Ali]$.
	Assume that $\Ao\succeq 0$, $(\Ao)_{mj}=(\Ao)_{jm}=0$ for any $m\neq j$, $m\in M$, and that there exists a diagonal matrix $D$ whose terms satisfy
	\begin{equation}
		(D)_{ii} D_{jj} (F_0)_{ij} \leq 0\qquad \forall i,j\notin M.
	\end{equation}
	Then, the \sproperty\ holds.
\label{theorem:sproperty-result2}
\end{theorem}

\section{Existence of solutions}
\label{sec:existence}
We introduce the following set of dual variables:
\begin{equation}
	W = \set{\blam \in \Gamma \,|\, \Ao+ \smsum_i \lambda_i \Ai \succeq 0}.
\label{eq:def-W}
\end{equation}

We have the following proposition:
\begin{proposition}
Function $q(\blam)$ is coercive if \cref{assumption:slater} is satisfied.
\label{theorem-coercive-existence}
\end{proposition}
\begin{proof}
See \cref{appendixB}.
\end{proof}

\begin{corollary}
Suppose \cref{assumption:slater} is satisfied and there exists $\blam$ such that $\Ao+\smsum_i \lambda_i \Ai\succ 0$. Then, $q(\blam)$ has a nonempty and compact solution set.
Furthermore, \ref{QPl} and \ref{QPe} have a nonempty solution set.
\end{corollary}
\begin{proof}
Using the Weirstrass theorem~\cite[Prop.3.2.1]{Bertsekas2009}, $q(\blam)$ is coercive and has non-empty interior, which guarantees existence of a non-empty and compact solution set.
This guarantees boundedness of the solution.
To satisfy attainability we can show that every point in the quadratic problem is regular.
This occurs if some constraint qualification is satisfied in the feasible set.
Since we have separable constraints, the linear independece qualification constraint (LICQ) are automatically satisfied at every point.
Therefore, every stationary solution must satisfy a KKT condition and such points are attained.
\end{proof}

\section{Algorithms}
\label{sec:algorithms}

Based on semidefinite programming the QCQP can be rewritten as
\begin{equation}
\begin{IEEEeqnarraybox}[][c]{r'l}
	\min_{\x}   & \trace[\Ao X] + 2 \bo ^T \x + \co \\
	\text{s.t.} & \trace[\Ai X] + 2 \bi ^T \x + \cli \unlhd 0,\quad \forall i\in \N, \\
				& X = \x \x^T.
\end{IEEEeqnarraybox}
\end{equation}
Relaxing the rank constraint on $X$ to $X\succeq \x \x ^T$, and formulating the constraint using the Schur complement (see A.5.5 in \cite{Boyd2004}), we get
\begin{equation}
\begin{IEEEeqnarraybox}[][c]{r'l}
	\min_{\x}   & \trace[\Ao X] + 2 \bo ^T \x + \co \\
	\text{s.t.} & \trace[\Ai X] + 2 \bi ^T \x + \cli \unlhd 0,\quad \forall i\in \N, \\
	            & \begin{pmatrix}
		          	X & \x \\
					\x^T & 1 
				  \end{pmatrix} \succeq 0,
\end{IEEEeqnarraybox}
\label{prob:QPl-sdp}
\end{equation}
which is a convex problem that can be solved using a semidefinite program solver.
Problem \ref{prob:QPl-sdp} yields an optimal solution of problems \ref{QPl} or \ref{QPe}, if conditions of \cref{theorem:sproperty-result} are satisfied.

Likewise, solving problem \ref{eq:dual-problem-QPl} as the dual formulation also achieves an optimal solution.
The primal values can then be recovered determining
\begin{equation}
\label{eq:recovery-x}
\begin{IEEEeqnarraybox}[][c]{r'l}
	\argmin_{\x} & \x^T (\Ao +\smsum_i \lambda_i^{\ast}\Ai)\x + (\bo + \smsum_i \lambda_i^{\ast} \bi )^T \x,
\end{IEEEeqnarraybox}
\end{equation}
where $\blam^{\ast}$ corresponds to the solution of problem \ref{eq:dual-problem-QPl}.
In particular, the solution belongs to the following subspace
\begin{equation}
	\x^\ast \in -\big(\Ao + \smsum_i \lambda_i^{\ast} \Ai\big)^\dagger \big(\bo + \smsum_i \lambda_i^{\ast} \bi\big)^T +\nullspace\big[\Ao + \smsum_i \lambda_i^{\ast} \Ai\big],
\end{equation}
which has to be feasible and satisfy the complementarity conditions:
\begin{IEEEeqnarray}{l"l}
	\label{eq:cs-sdp}
	\lambda^{\ast}_i g_i(x_i^{\ast} ) = 0 & \forall i\in\N.
\end{IEEEeqnarray}
Note that if $\Ao + \smsum_i \lambda_i^{\ast} \Ai\succ 0$ the solution is unique.

Equation \eqref{eq:recovery-x} always has a solution if $\blam^{\ast}$ exists, since \eqref{eq:recovery-x} is an unconstrained quadratic problem.
In the case in which \eqref{eq:recovery-x} has multiple solutions, it is possible that not all of them will satisfy \eqref{eq:cs-sdp}.
How to discover a feasible solution in this case is studied in \cite{Fortin2004}.

\subsection{Projected dual ascent algorithm}
\label{sub:dual-gradient}
Solving a semidefinite program as \ref{eq:dual-problem-QPl} or \ref{prob:QPl-sdp} is costly, where interior point methods have a computational complexity of $O(p^5)$ in a worst-case setting to $O(p^3)$ for specialized structured methods, as described in \cite{Vandenberghe2005}.
This motivates to use other methods that may yield better performance in larger settings.

We introduce some new notation, and we denote $\g(\x)=(g_i(\x))_{i\in\N}$ as a column operator that collects all of the constraints.
The ascent algorithm is based on the following problem
\begin{IEEEeqnarray}{r'l}
	\max_{\blam} & \blam^T \g(\tilde{\x}(\blam^{k},\x^k)) \\
	\text{s.t.}  & \blam\in W,
\label{eq:dual-problem-QP}
\end{IEEEeqnarray}
where $W$ is defined in \eqref{eq:def-W} and $\widetilde{\x}(\blam^{k},\x^k)$ is an operator defined as:
\begin{equation}
	\widetilde{\x}(\blam,\x ^k) = \argmin_{\x\in\mathbb{R}^p}\: L(\x,\blam) 
\label{eq:argmin-lagrangian-qp}
\end{equation}
where $L(\x,\blam)$ is defined in \eqref{eq:lagrangian-QP}.
The projected dual ascent method is described in \cref{alg:projected-dual-ascent}.
Convergence analysis and performance corresponds to that of subgradient methods, see \cite{Boyd2003}.
In our formulation we introduced $\Pi_{W}$ as a projection operator onto set $W$ and $\alpha^k$ as the step size at iteration $k$.
The projection onto $W$ can be done using an L2 norm:
\begin{equation}
	\Pi_W(\blam^{\text{temp}})  = \argmin_{\blam\in W}\: \sum_{i\in \N} \Vert \lambda_i-\lambda_i^{\text{temp}} \Vert^2 .
\label{eq:projection-onto-W-L2}
\end{equation}

\begin{algorithm}
\begin{algorithmic}[1]
	\State Initialize $\x^0$ and $\blam^0$. Set $k\leftarrow 0$.
	\While{$\blam^k$ is not a stationary point}
		\State $\x^{k+1} = \widetilde{\x}(\blam^{k},\x^k)$. \label{eq:xg-operator} 
		\State $\blam^{\text{temp}} =  \blam^{k} + \mu_k \g(\x^{k+1})$.
		\State $\blam^{k+1} =  \Pi_{W}[\blam^{\text{temp}}]$.
		\State Set $k\leftarrow k+1$
	\EndWhile
\end{algorithmic}
\caption{Projected dual subgradient method}
\label{alg:projected-dual-ascent}
\end{algorithm}

\subsubsection{Augmented Lagrangian}
In order to infuse stability to dynamical systems derived from subgradient methods, it is frequent to use an augmented Lagrangian formulation \cite{Boyd2010}.
By doing so, the augmented Lagrangian becomes differentiable (if the problem becomes strictly convex) and the subgradient becomes a gradient at every iteration point.
In such case, problem \eqref{eq:argmin-lagrangian-qp} is substituted by:
\begin{equation}
	\widetilde{\x}(\blam,\x ^k) = \argmin_{\x\in\mathbb{R}^p}\: L(\x,\blam) + \sum_{i\in\N } \rho c( g_i(x_i) ).
\label{eq:augmented_lagrangian_centralized}
\end{equation}
where $c(\cdot)$ corresponds to a convex function that does not alter the cost of the problem at the optimal solution.
Assuming that $g_i(x_i)$ is part of an equality constraint, $c(\cdot) = \Vert \cdot \Vert ^2$ suffices to enforce stability to the subgradient method.
Alternatively, if $g_i(x_i)$ is part of an inequality constraint, then
\begin{equation}
	c(x) = 
	\begin{cases}
		\Vert x \Vert^2 & \text{if } x>0 \\
		0 & \text{otherwise}
	\end{cases}
\end{equation}
achieves the same purpose.

It is easy to verify that the squared term added to problem \eqref{eq:argmin-lagrangian-qp} does not alter the solution points nor the optimality requirements of the original quadratic problem.

\subsection{Projected dual ascent method with FLEXA decomposition}
\label{sub:distributed}
A centralized method such as \cref{sub:dual-gradient} can be parallelized solving the inner problem \eqref{eq:argmin-lagrangian-qp} using parallel techniques.
The projection problem remains a centralized step, which needs to be computed globally.
The proposed decomposition will allow to handle problems with large number of variables, or distributed datasets, which currently constitute problems of high interest.
To the best of our knowledge, there exists no framework for nonconvex QCQPs that uses a distributed or parallel scheme, which guarantees optimality even under the strong duality property.
In our proposal, we extend the setting of \cref{sub:dual-gradient} to include parallel updates using the FLEXA decomposition theory presented in \cite{Scutari2014}.

The methodology consists of a two-loop procedure that emerges from the dual ascent method presented in \cref{sub:dual-gradient}.
Starting from a dual formulation guarantees that the method will converge to the optimal solution whenever strong duality holds.
This is not satisfied if the algorithmic method relies on descent techniques over the primal problem, as the procedure may converge to a local optima.

The joint optimization update \eqref{eq:argmin-lagrangian-qp} is decoupled and solved iteratively using the FLEXA framework from \cite{Scutari2014}, while the dual variables are updated in an outer loop.
The framework approximates at every step the original problem with a convex surrogate that can be decomposed in parallel problems and solved efficiently.
The approximation is formed fixing the non-local group of variables to a point from a previous iteration, and optimizing only in the local variables.
The separability of the constraints allows to use a natural decomposition in block of variables.
Furthermore, a form that uses the augmented Lagrangian presented in \eqref{eq:augmented_lagrangian_centralized} can also be used, improving the convergence properties of the dual method if the unaugmented Lagrangian is not strongly convex in $\x$.

Following our procedure and satisfying the conditions from Theorem 3 in \cite{Scutari2014}, then the iterates will converge to the global minimum point of the Lagrangian \eqref{eq:argmin-lagrangian-qp}.
This occurs because the Lagrangian is convex on $\x$, and the descent steps are able to converge to the global minimum.

We start by proposing a surrogate problem, which has the following form:
\begin{equation}
\begin{IEEEeqnarraybox}[][c]{rCl}
	U(\x,\blam;\x^{k,q}) & = & \sum_{i\in\N} U_i(x_i,\blam;\x^{k,q}) + \rho_i \Vert x_i - x_i^{k,q} \Vert^2,  
\end{IEEEeqnarraybox}
\label{eq:parallel-problem}
\end{equation}
where
\begin{equation}
	U_i(x_i,\lambda_i;\x^{k,q}) = x_i^T \Aoii x_i + 2 \smsum_{j\neq i} (x_j^{k,q})^T \Aoij x_i + 2 \boi^T x_i + \lambda_i (x_i^T \Ali x_i + 2 \bli^T x_i + \cli),
\end{equation}
$\x^{k,q} = (x_j^{k,q})_{j\in\N}$ is a pivotal point indexed by the outer loop $k$, and inner loop $q$.
Matrix $\Ao$ is decomposed as follows:
\begin{equation}
	\Ao = 
	\begin{pmatrix}
		A_{0}^{11} & \cdots & A_0^{1N} \\
		\vdots & \ddots & \vdots \\
		A_0^{N1} & \cdots & A_0^{NN}
	\end{pmatrix}.
\end{equation}
Parameter $\rho_i \geq 0$ can be set to zero if $U_i(x_i,\lambda_i;\x^{k,q})$ is strongly convex in $x_i$.

Note that \eqref{eq:parallel-problem} is separable and convex in $x_i$.
It is a quadratic problem with no constraints, whose solution can be computed as follows (Appendix A in~\cite{Boyd2004}):
\begin{IEEEeqnarray}{l}
	\min_{x_i\in\mathbb{R}^{n_i}} \: U_i(x_i,\lambda_i;\x^{k,q}) + \rho_i \Vert x_i - x_i^{k,q} \Vert^2  \label{eq:qp-flexa-local} \\
	= \begin{cases}
		- r(\x^{k,q})^T X_i(\lambda_i) r(\x^{k,q}) + \co + \lambda_i \ci 
		& \text{if } X_i(\lambda_i) \succeq 0\:\wedge\: r(\x^{k,q}) \in \range[X_i(\lambda_i)]   \\
		-\infty & \text{otherwise} \IEEEnonumber
	\end{cases}
\label{eq:dual_problem_flexa}
\end{IEEEeqnarray}
where $X_i(\lambda_i) = \Aoii + \lambda_i \Ali + \rho_i \id_{n_i}$ and $r(\x^{k,q})=\boi+\lambda_i\bli+\sum_{j\neq i}(\Aoij)^T x^{k,q}_{j}+\rho_i x_i^{k,q}$.
The solution comes as
\begin{equation}
\begin{IEEEeqnarraybox}[][c]{rCL}
	\widetilde{x}_i(\lambda_i;\x^{k,q}) & = & \argmin_{x_i\in\mathbb{R}^{n_i}}\: U_i(x_i,\lambda_i;\x^{k,q}) + \rho_i \Vert x_i - x_i^{k,q} \Vert^2 \\
	& = & - X(\lambda_i)^{\dagger} r(\x^{k,q}) + \nullspace[X(\lambda_i)].
\end{IEEEeqnarraybox}
\end{equation}
Because the quadratic problem is strongly convex if $\rho_i>0$, there is a unique solution to the problem.
We can set $\rho_i = 0$ if $U_i(x_i,\lambda_i;\x^{k,q})$ is strongly convex.
All steps that we have discussed are described in \cref{alg:qp-flexa-dual-ascent}.

\begin{algorithm}
\begin{algorithmic}[1]
	\State Initialize $\x^{0,0}$ and $\blam^0$. Set $k\leftarrow 0$.
	\While{$\blam^k$ is not a stationary point}
		\State Set $q\leftarrow 0$.
		\While{$\x^{k,q}$ is not a stationary point}
			\State Compute $\widetilde{x_i}(\lambda_i,\x^{k,q})$ for all $i\in\N$ \textbf{in parallel}.
			\State Update $x_i^{k,q+1} = x_i^{k,q} + \alpha_q (\widetilde{x_i}(\lambda_i^k;\x^{k,q}) - x_i^{k,q})$ for all $i\in\N$.
			\State Set $q \leftarrow q+1$.
		\EndWhile
		\State $\blam^{\text{temp}} =  \blam + \mu_k \g(\x^{k,q})$.
		\State $\blam^{k+1} =  \Pi_{W}[\blam^{\text{temp}}]$.
		\State Set $k\leftarrow k+1$
	\EndWhile
\end{algorithmic}
\caption{Dual gradient ascent method with distributed FLEXA decomposition.}
\label{alg:qp-flexa-dual-ascent}
\end{algorithm}

Convergence of \cref{alg:qp-flexa-dual-ascent} is guaranteed if conditions of Theorem 3 in \cite{Scutari2014} are fulfilled and $\mu_k$ are chosen small enough so that the (sub)gradient method converges.
Combining the convergence results from \cite{Scutari2014} and diminishing step sizes from \cite{Boyd2003} for subgradient methods, we can establish the following theorem:
\begin{theorem}
	Given problem \ref{QPl} or \ref{QPe}. Assume that the \sproperty\ holds,
	\begin{equation*}
		0 < \inf_q \alpha_q \leq \sup_q \alpha_q \leq \alpha_{\max} \leq 1,\quad \rho_i>0\quad \forall i\in\N,
	\end{equation*}
	and $\mu_k$ is small enough. Then, \cref{alg:qp-flexa-dual-ascent} converges to an optimal solution of the QP.
\end{theorem}

An extension of this decomposition using the augmented Lagrangian formulation from \cref{sub:dual-gradient} is possible, although in this case problem \eqref{eq:qp-flexa-local} is no longer quadratic and must be solved using some convex optimization method, rather than using expression \eqref{eq:dual_problem_flexa}.

\section{Robust least squares}
\label{sec:rls}

Least squares (LS) is a popular method that finds linear dependencies between output data $y$ and a list of input observations given by $A$ and $b$, such that, in the absence of noise, $A x -b = y$.
We assumed $x\in\mathbb{R}^p$, $A$ is matrix of dimensions $N\times p $ and $b\in\mathbb{R}^N$.
In the presence of noise, a squared error measure is proposed and minimized, such that $\Vert Ax-b\Vert^2$ has the smallest value possible.
The total least squares (TLS) problem aims to minimize an error, allowing changes in matrix $A$ and vector $b$, such that $(A+\Delta A)x -(b+\Delta b) = 0$ with minimal energy of $\Vert( \Delta A, \Delta b)\Vert^2_{\mathcal{F}}$.

Both problems are widely used in the literature, but they may present high sensitivity to perturbations in the feature components.
For this reason, regularization alternatives have been proposed that aim to minimize the output variance under noisy input parameters.
An alternative to such methods proposed in \cite{ElGhaoui1997} consists of solving a robust least squares (RLS) version as follows:
\begin{equation}
\begin{IEEEeqnarraybox}[][c]{r'l}
	\min_{x\in\mathbb{R}^p} \max_{(\Delta A,\Delta b)} & \Vert (A+\Delta A)x - (b+\Delta b)\Vert^2 \\
	\text{s.t.} & \Vert (\Delta A,\Delta b) \Vert^2 \leq \rho,
\end{IEEEeqnarraybox}
\end{equation}
where the output $x$ of such problem corresponds to the robust solution under finite perturbations on values $A$, $b$.

Our proposal to illustrate our nonconvex quadratic results consists of solving the RLS problem whenever the error bound $\rho$ is specified for different groups of values belonging to $A$ and $b$.
The reason behind different error values may be due to a distributed source of data, such as in \cite{Chen2015}.
Our RLS proposal becomes:
\begin{equation}
\begin{IEEEeqnarraybox}[][c]{r'l'l}
	\min_{x\in\mathbb{R}^p} \max_{(\Delta A,\Delta b)} & \Vert (A+\Delta A)x - (b+\Delta b)\Vert^2 \\
	\text{s.t.} & \Vert (\Delta A)_{:i} \Vert^2 \leq \rho_i & \forall i\in \set{1,\ldots,p} \\
	            & \Vert \Delta b \Vert^2 \leq \rho_{p+1},
\end{IEEEeqnarraybox}
\label{eq:rls-estimator}
\end{equation}
when errors are grouped in columns.
We denoted the $i$th column of matrix $\Delta A$ with $(\Delta A)_{:i}$.
We remark that the methods proposed in \cite{ElGhaoui1997} are only valid for a single quadratic constraint, but we can use the analysis we developed to solve the problem with multiple constraints as long as the \sproperty\ is satisfied.

Methods described by \cite{Ben-Tal2009} to solve robust problems significantly alter the structure of the QCQPs and prevent the use of our framework directly.
Therefore, we discuss an alternative procedure where we establish strong duality on every non-convex subproblem and provide convergence guarantees.

\subsection{Centralized proposal}
Our algorithmic proposal consists of a gradient descent method on variable $x$, while solving the maximization problem either in a centralized manner or using the decentralized method from \cref{sub:distributed}.
The specific steps for the minimization step and maximization problem are illustrated in \cref{alg:rls-sdp}.

The minimization step on variable $x$ takes the following form
\begin{equation}
	x^{k+1} = x^k + \alpha_x^k 2 (A+\Delta A)^T ((A+\Delta A)x^k-(b+\Delta b))
\end{equation}
where $\alpha_x^k$ corresponds to the gradient's step size and $k$ denotes the iteration step.

Consider the following renaming of variables: $\Delta = (\Delta A, \Delta b)$, $H = (A,b)$ and $\overline{x}=(x^T,-1)^{T}$ for a given $x$. 
The maximization problem can be written as
\begin{equation}
\begin{IEEEeqnarraybox}[][c]{r'l'l}
	\max_{\Delta} & \Vert (H + \Delta) \overline{x} \Vert^2 \\
	\text{s.t.} & \Vert \Delta_{:i} \Vert^2 \leq \rho_i & \forall i\in \set{1,\ldots,p+1},
\end{IEEEeqnarraybox}
\label{eq:rls-qp}
\end{equation}
where $\Delta_{:i}$ stands for the $i$th column of matrix $\Delta$.
By further denoting $U = \Delta^T \Delta$ and some transformations, an equivalent formulation becomes:
\begin{equation}
\begin{IEEEeqnarraybox}[][c]{r'l}
	\max_{U,\Delta} & \trace[U \overline{x}\overline{x}^T] + \trace[(H^T\Delta+\Delta^T H)\overline{x}\overline{x}^T]+ \trace[H^T H \overline{x}\overline{x}^T]\\
	\text{s.t.} & U_{ii} = \rho_i \qquad \forall  i\in \set{1,\ldots,p+1} \\
	            & U = \Delta^T \Delta.
\end{IEEEeqnarraybox}
\end{equation}
Easily, the SDP relaxation is obtained as follows when $U \succeq \Delta^T \Delta$, i.e.,
\begin{equation}
\begin{IEEEeqnarraybox}[][c]{r'l}
	\max_{U,\Delta} & \trace[U \overline{x}\overline{x}^T] + \trace[(H^T\Delta+\Delta^T H)\overline{x}\overline{x}^T]+ \trace[H^T H \overline{x}\overline{x}^T]\\
	\text{s.t.} & U_{ii} = \rho_i \qquad \forall  i\in \set{1,\ldots,p+1} \\
	            & 
	            \begin{pmatrix}
	            	U & \Delta^T \\
	            	\Delta & \id_{N}	           		
	            \end{pmatrix} \succeq 0.
\end{IEEEeqnarraybox}
\label{eq:rls-sdp}
\end{equation}

We can prove that the \sproperty\ is satisfied and that strong duality holds for this problem. 
This is achieved in the next theorem:

\begin{theorem}
Assume $\rho_i>0$ for all $i$. 
Then, strong duality between primal problem~\eqref{eq:rls-qp} and its dual~\eqref{eq:rls-sdp} holds for any $H\in \mathbb{R}^{N\times p+1}$ and $\overline{x}=(x^T,-1)^T\in\mathbb{R}^{p+1}$.
\label{th:rls}
\end{theorem}
\begin{proof}
See \cref{appendixC}.
\end{proof}

The centralized method is described in \cref{alg:rls-sdp}.
\begin{algorithm}
\begin{algorithmic}[1]
	\State Initizalize $x^0$, $\Delta^0$ and $\{\alpha^k\}$.
	\While{STOP criteria is not satisfied on $x^k$}
		\State Solve problem \eqref{eq:rls-sdp}. Assign result to $\Delta^k = (\Delta A^k, \Delta b^k)$.
		\State Set $x^{k+1} \leftarrow x^k + \alpha_x^k 2 (A+\Delta A^k)^T ((A+\Delta A^k)x^k-(b+\Delta b^k))$.
		\State Set $k\leftarrow k+1$.
	\EndWhile
\end{algorithmic}
\caption{Centralized RLS algorithm}
\label{alg:rls-sdp}
\end{algorithm}

\appendix
\section{Proofs of \cref{sec:main}}
\label{appendixA}

\begin{proof}[\cref{lemma:strong-alternatives}]
From \cref{def:sproperty-qpl}, the \sproperty\ holds if the following systems are strong alternatives:
\begin{IEEEeqnarray}{l} \IEEEyesnumber
	\exists \x\in\mathbb{R}^p\,|\,f(\x) < \alpha\,\wedge\,g_i(x_i) \unlhd 0 \quad\forall i\in\N,\quad \Longleftrightarrow  
	\IEEEyessubnumber \\
	\nexists \blam\in\Gamma\,|\,f(\x)+\sum_{i\in\N}\lambda_i g_i(x_i)\geq \alpha\quad	\forall \x\in\mathbb{R}^p. 
	\IEEEyessubnumber \label{eq:sa-22}
\end{IEEEeqnarray}
We need to prove the equivalence of \eqref{eq:sa-21} and \eqref{eq:sa-22}.
The implication \eqref{eq:sa-21} $\Rightarrow$ \eqref{eq:sa-22} is immediate from
\begin{equation}
	[x,1]\Big(\Abco+\sum_{i\in\N} \lambda_i \Abci -\alpha \evec \Big)[x,1]^T=f(\x)+\sum_{i\in\N} \lambda_i g_i(x_i)- \alpha\geq 0.
\end{equation}
The implication \eqref{eq:sa-21} $\Leftarrow$ \eqref{eq:sa-22} is shown as follows. First, we define $h(\x,\blam)$:
\begin{IEEEeqnarray}{rCl}
	h(\x,\blam) & = & f(\x)+\sum_{i\in\N}\lambda_i g_i(x_i)- \alpha \nonumber \\
	 			& = & \x^T\Ao \x+2\bo^T\x+\co-\alpha+\sum_{i\in\N}\lambda_i\big(\x^T\Ai\x+2\bi^T\x+\cli\big) \geq 0.
\end{IEEEeqnarray}
Function $h(\x,\blam)$ is nonnegative $\forall \x\in\mathbb{R}^p$ if and only if a minimum exists and it is nonnegative, and if the function is convex.
This occurs if and only if
\begin{IEEEeqnarray}{l} \IEEEyessubnumber \phantomsection \label{eq:pos-sproperty}
	\Ao+\smsum_{i} \lambda_i\Ai\succeq 0 
	\IEEEyessubnumber \label{eq:pos-sproperty1}\\
	\co+\smsum_{i} \lambda_i\cli-\alpha-(\bo+\smsum_{i} \lambda_i\bli)^T \big(\Ao+\smsum_{i}\lambda_i\Ai\big)^{\dagger}(\bo+\smsum_{i}\lambda_i\bli)\geq 0
	\IEEEyessubnumber \label{eq:pos-sproperty2} \\
	\bo+\lambda_i \bi \in\range[\Ao+\smsum_{i} \lambda_i\Ai] 
	\IEEEyessubnumber \label{eq:pos-sproperty3}.
\end{IEEEeqnarray}
are all satisfied.
Note that \eqref{eq:pos-sproperty1} corresponds to the Hessian of $h(\x,\blam)$,
\eqref{eq:pos-sproperty2} corresponds to $h(\x_{\min},\blam)$, where
$\x_{\min}=-\big(\Ao+\sum_i\lambda_i\Ai\big)^{\dagger}(\bo+\sum_i\lambda_i\bli)$ is the minimum value of $h(\x,\blam)$,
and \eqref{eq:pos-sproperty3} guarantees there is no direction where the problem is unbounded.
See Appdx. A.5.4 in \cite{Boyd2004} for the complete description of the solution of a QCQP.
Given \eqref{eq:pos-sproperty} and using the relation of the Schur complement of a matrix and its positive semidefiniteness as presented in page 651 in \cite{Boyd2004}, we conclude that $\Abco+\sum_{i} \lambda_i \Abci -\alpha \evec\succeq 0$.
This ends the proof of the equivalence.
\end{proof}

\begin{proof}[\cref{lemma:sdp-strong-alternatives}]
The proof is a standard result in semidefinite programming and can be derived in a very similar way to Example 5.14 in \cite{Boyd2004}. 
We reproduce it here for completeness.
Consider the following problem:
\begin{equation}
\begin{IEEEeqnarraybox}[][c]{r'l}
	\min_{s\in\mathbb{R},\blam\in\Gamma} & s  \\
	\hspace{1.2em}\text{s.t.} & \Abco+\sum_{i\in\N}\lambda_{i}\Abci-\alpha\evec \succeq-s\id_{p+1}.
\end{IEEEeqnarraybox}
\label{eq:sdp-primal-problem}
\end{equation}
The Lagrangian has the following form:
\begin{IEEEeqnarray}{rCl}
	L(s,\blam,Z) & = & s-\trace\big[ (\Abco+\sum_{i\in\N}\lambda_{i}\Abci-\alpha\evec+s\id_{p+1} )Z\big] \nonumber \\
				 & = & s-\trace[(\Abco-\alpha\evec) Z]-\sum_{i\in\N}^{N}\lambda_{i}\trace[\Abci Z]-s\trace[Z]
\end{IEEEeqnarray}
and the dual function can take the following values:
\begin{equation}
\begin{IEEEeqnarraybox}[][c]{r'l}
	\min_{s\in\mathbb{R},\blam\in\Gamma} & L(s,\blam,Z)
	= \begin{cases}
			-\trace[(\Abco-\alpha\evec) Z] & \text{if }\trace[\Abci Z] \unlhd 0\;\forall i,\;\trace[Z]=1\\
			-\infty & \text{otherwise}.
	 \end{cases}
\end{IEEEeqnarraybox}\label{eq:sdp-dual-function}
\end{equation}
Maximizing the dual function given by \eqref{eq:sdp-dual-function} we obtain the dual problem:
\begin{equation}
\begin{IEEEeqnarraybox}[][c]{r'l}
	\max_{Z}    & -\trace[(\Abco-\alpha\evec) Z] \\
	\text{s.t.} & \trace[\Abci Z] \unlhd\;\forall i \\
			    & Z\succeq 0,\quad \trace[Z]=1.
\end{IEEEeqnarraybox}
\label{eq:sdp-dual-problem}
\end{equation}
Primal problem \eqref{eq:sdp-primal-problem} is convex and has non-empty interior (it satisfies Slater's condition) and, therefore, strong duality holds with it's dual formulation given by \eqref{eq:sdp-dual-problem}.

A technical condition needs to be satisfied to establish the strong alternative property between the systems described in 
\eqref{eq:sdp-strong-alternative-systems}: problem \eqref{eq:sdp-primal-problem} has to attain its infimum.
A sufficient condition is given by the following implication (see Example 5.14 in \cite{Boyd2004}):
\begin{equation}
	\nu_i\in\Gamma ,\:\sum_{i\in\N}\nu_i \Abci\succeq 0\Longrightarrow \nu_i=0\quad(\forall i\in\N) \Longrightarrow \sum_{i\in\N} \nu_i \Abci =0
\label{eq:sdp-constraint-qualification}
\end{equation}
which in our case is satisfied because of \cref{assumption:slater}.
In particular, every matrix $\Abci$ has at least a negative eigenvalue in the case of \ref{QPl}, and positive and negative eigenvalues in the case of \ref{QPe}, because of Slater's condition.
Moreover, matrices $\Abci$ do not overlap between each other except for the matrix elements $(p+1,p+1)$.
As a consequence, any positive linear combination of these matrices will have at least a negative eigenvalue, which fulfills the condition described in \eqref{eq:sdp-constraint-qualification}.

It is now clear that primal problem \eqref{eq:sdp-primal-problem} and dual problem \eqref{eq:sdp-dual-problem} attain the same objective value.
Therefore, if $s>0$, then $\trace[(\Abco-\alpha\evec) Z]<0$ and there exists no $\blam\in\Gamma$ such that $\Abco+\sum_{i\in\N}\lambda_i \Abci -\alpha \evec \succeq 0$. 
This shows that both conditions are strong alternatives and it
concludes the proof of the lemma.
\end{proof}

\begin{proof}[\cref{lemma:from-sdp-to-QPl}]
Implication \eqref{eq:vector-to-equation-1} $\Longleftarrow$ \eqref{eq:vector-to-equation-2} is verified, since we can introduce $z=(\x^T,1)^T=Py$ and satisfy \eqref{eq:vector-to-equation-1}:
\begin{IEEEeqnarray}{l} \IEEEyesnumber \phantomsection \label{eq:zy-domain}
    f(\x)-\alpha = z^T (\Abco-\alpha \evec) z = y^T P^T (\Abco-\alpha \evec) P y = y^T F_0 y < 0 
    \IEEEyessubnumber \\
	g_i(\x) = z^T \Abci z = y^T P^T \Abci P y = y^T F_i y \unlhd 0,\quad \forall i\in\N. 
	\IEEEyessubnumber
\end{IEEEeqnarray}

The opposite direction \eqref{eq:vector-to-equation-1} $\Longrightarrow$ \eqref{eq:vector-to-equation-2} is more elaborate.
We first need to show the following equivalence:
\begin{IEEEeqnarray}{l} \IEEEyesnumber \phantomsection \label{eq:vector-to-equation-F}
	\exists y\in\mathbb{R}^{p+1}\,|\, y^T F_0 y<0\,\wedge\,y^T \Fi y\leq 0,\,\forall i\in\N \quad \Longleftrightarrow 
	\IEEEyessubnumber \label{eq:vector-to-equation-F1}\\
	\exists \s\in\mathbb{R}^p\,|\,\s^T \Fo \s+2\fo^T \s+\dobj < 0\,\wedge\,\s^T \Fi \s + \ei \leq 0\quad \forall i\in\N,
	\IEEEyessubnumber \label{eq:vector-to-equation-F2}
\end{IEEEeqnarray}
where $\Fo$, $\fo$, $\dobj$, $\Fi$ and $\ei$ are defined in \eqref{eq:Fo-Fi-2-QP}.

Consider $y=(v^T,w)^T$ with $v\in\mathbb{R}^p$ and $w\in\mathbb{R}$. We have
\begin{equation}
	\begin{pmatrix}
		v \\ w
	\end{pmatrix}^T
	\underbrace{\begin{pmatrix}
		\Fo & \fo \\
		\fo^T & \dobj
	\end{pmatrix}}_{F_0}
	\begin{pmatrix}
		v \\ w
	\end{pmatrix} <0,\quad
	\begin{pmatrix}
		v \\ w
	\end{pmatrix}^T
	\underbrace{\begin{pmatrix}
		\Fi & 0 \\ 
		0^T  & \ei
	\end{pmatrix}}_{F_i}
	\begin{pmatrix}
		v \\ w
	\end{pmatrix} \leq 0,\quad \forall i\in\N.
\label{eq:Fo-Fi-2-QP}
\end{equation}
If $w\neq 0$ we can choose $\s=v/w$ and \eqref{eq:vector-to-equation-F2} is feasible.
If $w = 0$ we have $v^T\Fo v<0$ and $v^T \Fi v\leq 0$.
We can choose $\s=\sum_{i\in\N} \ts \hat{s}_i+ \tv v$, where $\hat{\s}$ is a  point satisfying Slater's condition (see \cref{assumption:slater}) in the congruent domain.
We define $\hat{\s}$ such that $\hat{\s}=\sum_i \hat{s}_i$ and the support of all $\hat{s}_i$ do not overlap.
This results into the following:
\begin{IEEEeqnarray}{l} \IEEEyesnumber
	\s^{T}\Fo  \s+2\fo^T \s+\dobj  \IEEEyessubnumber \label{eq:f(x)<0} \\
	\hspace{2em} 
	= (\sum_i \ts \hat{s}_i)^{T} \Fo (\sum_i \ts \hat{s}_i)+2\fo^T (\sum_i \ts \hat{s}_i)+\dobj +\underbrace{\tv^2 v^T \Fo v}_{<0}
	+2\tv(\sum_i \ts \Fo \hat{s}_i+\fo)^T v \IEEEnonumber \vspace{0.75em}  \\ 
	\IEEEeqnarraynumspace
	 \s^T \Fi \s+ \ei = \underbrace{\ts^2 \hat{s}^T_i \Fi \hat{s}_i +\ei}_{<0}
	 + \underbrace{\tv^2 v^T \Fi v}_{\leq0} + 2\tv \ts \hat{s}^T_i \Fi v  
	 < 2\tv \ts \hat{s}_i^T \Fi v \IEEEyessubnumber
\end{IEEEeqnarray}
By taking $\tv\rightarrow \infty$ and choosing $\ts=\pm 1$ depending on the signs of $\hat{s}^T_i \Fi v $, we can find an $\s$ such that \eqref{eq:vector-to-equation-F2} is satisfied.
In the original domain, we can recover $\x$ with $(\x,1)^T = P\cdot (\s,1)^T$, which satisfies \eqref{eq:vector-to-equation-2}.
This completes the proof.
\end{proof}

\begin{proof}[\cref{theorem:rank1}]
Because $Y\succeq 0$, it admits an orthogonal decomposition such that $Y=VV^T$ where $V=[v_1,v_2,\cdots, v_R]$ and $R\leq p+1$ represents the rank of $Y$.
Then,
\begin{IEEEeqnarray}{rCl}
	\trace[F_0 Y] & = & \trace[V^T F_0 V] = \sum_{r=1}^R v_r^T F_0 v_r 
	= \sum_{r=1}^R \sum_{i=1}^{p+1} (F_0)_{ii} v^2_{ri} +2\sum_{j>i} (F_0)_{ij} v_{ri} v_{rj}  <0 \\
	\trace[F_i Y] & = & \trace[V^T F_i V] = \sum_{r=1}^R v_r^T F_i v_r = \sum_{r=1}^R \sum_{j=1}^{p+1} (F_i)_{jj} v^2_{rj} = \psi_i\quad \forall i\in\N,
\end{IEEEeqnarray}
where $v_{rj}$ corresponds to the $j$'th component of vector $v_r$.
Recall that $(F_i)_{ij}=0$ for $i\neq j$. 
We consider vector $y=(y_j)_{j=1}^{p+1}$ such that $y_j^2 = \sum_{r=1}^Rv^2_{rj}$.
This implies that
\begin{equation}
	y^T F_i y = \sum_{j=1}^{p+1} (F_i)_{jj} y_j^2 =\sum_{j=1}^{p+1} (F_i)_{jj} \sum_{r=1}^R v^2_{rj} = \sum_{r=1}^R \sum_{j=1}^{p+1} (F_i)_{jj}v^2_{rj}
	= \trace[F_i Y] = \psi_i.	
\end{equation}
We introduce a diagonal matrix $D$ with $(D)_{jj}=\pm 1$ such that $y_j = (D)_{jj} \sqrt{\smsum_{r=1}^Rv^2_{rj}}$.
The product with $F_0$ becomes:
\begin{IEEEeqnarray*}{rCl}
	y^T F_0 y & = & \sum_{i=1}^{p+1} (F_0)_{ii} y_i^2 + 2\sum_{j>i} (F_0)_{ij} y_i y_j \\
	          & = & \sum_{i=1}^{p+1} (F_0)_{ii} \Big(\sum_{r=1}^R v^2_{ri}\Big) + 2\sum_{j>i} (F_0)_{ij} \bigg((D)_{ii} \sqrt{\sum_{r=1}^R v^2_{ri}}\bigg) 
	          \bigg((D)_{jj} \sqrt{\sum_{r=1}^R v^2_{ri}}\bigg)
\end{IEEEeqnarray*}
Consider the difference
\begin{equation}
	y^T F_0 y-\trace[F_0 Y]= 2\sum_{j>i} (F_0)_{ij} \bigg((D)_{ii} (D)_{jj} \sqrt{\smsum_{r,s=1}^R v^2_{ri}v^2_{sj}} - \smsum_{r=1}^R v_{ri} v_{rj}\bigg) \stackrel{?}{\leq}0
\label{eq:negative-difference}
\end{equation}
Expression \eqref{eq:negative-difference} is nonpositive if $(F_0)_{ij}(D)_{ii} (D)_{jj}$ is nonpositive for every pair $i\neq j$.
This is shown in \cref{prop:vector-inequality}.
\end{proof}

\begin{proposition}
	The following relation holds true for any sequence $v_{ri},v_{rj}$, $r\in\set{1,\cdots,R}$ and pair $i,j\in\set{1,\cdots,p+1}:$
\begin{equation}
	\sqrt{\smsum_{r,s=1}^R v^2_{ri}v^2_{sj}} \geq \smsum_{r=1}^R v_{ri} v_{rj}
	\label{eq:vector-inequality}
\end{equation}
\label{prop:vector-inequality}
\end{proposition}
\begin{proof}
Take squares on the left and right hand sides of \eqref{eq:vector-inequality}, which gives
\begin{equation}
	\sum_{r,s=1}^R v^2_{ri}v^2_{sj} \geq \sum_{r,s=1}^R v_{ri} v_{sj} v_{rj} v_{si} 
	\label{eq:vector-inequality-2}
\end{equation}
Note that $a^2+b^2 \geq 2ab$ for any $a,b\in\mathbb{R}$. Introduce $a=v_{ri}v_{sj}$ and $b=v_{rj}v_{si}$:
\begin{equation}
	v_{ri}^2v_{sj}^2+v_{rj}^2v_{si}^2 \geq 2 v_{ri}v_{sj}v_{rj}v_{si},
\end{equation}
which is satisfied for every pair of $r, s$. Summing all of these inequalities yields \eqref{eq:vector-inequality-2} which implies \eqref{eq:vector-inequality}.
\end{proof}

\begin{proof}[\cref{lemma:from-sdp-to-QPe}]
Implication \eqref{eq:vector-to-equation-3} $\Longleftarrow$ \eqref{eq:vector-to-equation-4} is immediate, since we can introduce $z=(\x^T,1)^T=Py$ and satisfy \eqref{eq:vector-to-equation-3} same as in \cref{lemma:from-sdp-to-QPl}.
The direction \eqref{eq:vector-to-equation-3} $\Longrightarrow$ \eqref{eq:vector-to-equation-4} is more elaborate and we need to show the following equivalence:
\begin{IEEEeqnarray}{l} \IEEEyesnumber \phantomsection \label{eq:vector-to-equation-F-QPe}
	\exists y\in\mathbb{R}^{p+1}\,|\, y^T F_0 y<0\,\wedge\,y^T \Fi y= 0,\,\forall i\in\N \Longleftrightarrow 
	\IEEEyessubnumber \label{eq:vector-to-equation-F1-QPe}\\
	\exists \s\in\mathbb{R}^p\,|\,\s^T \Fo \s+2\fo^T \s+\dobj < 0\,\wedge\,\s^T \Fi \s + \ei = 0,\,\forall i\in\N 
	\IEEEyessubnumber \label{eq:vector-to-equation-F2-QPe}.
\end{IEEEeqnarray}

Consider $y=(v^T,w)^T$ with $v\in\mathbb{R}^p$ and $w\in\mathbb{R}$. We have
\begin{equation}
	\begin{pmatrix}
		v \\ w
	\end{pmatrix}^T
	\begin{pmatrix}
		\Fo & \fo \\
		\fo^T & \dobj
	\end{pmatrix}
	\begin{pmatrix}
		v \\ w
	\end{pmatrix} <0,\quad
	\begin{pmatrix}
		v \\ w
	\end{pmatrix}^T
	\begin{pmatrix}
		\Fi & 0 \\ 
		0^T  & \ei
	\end{pmatrix}
	\begin{pmatrix}
		v \\ w
	\end{pmatrix} = 0,\quad \forall i\in\N.
\end{equation}
If $w\neq 0$ we can choose $\s=v/w$ and \eqref{eq:vector-to-equation-F2-QPe} is feasible.
If $w = 0$ we have $v^T\Fo v<0$ and $v^T \Fi v = 0$ for all $i$.
We require that $\Fi \neq 0$ so that $\hat{s}_i^T \Fi \hat{s}_i\neq 0$.
This is equivalent to require that $\Ai \neq 0$.
We try $\s=\smsum_i \ts \hat{s}_i+ \tv v$ where the support of every $\hat{s}_i$ does not overlap.
We get the following:
\begin{IEEEeqnarray}{l} \IEEEyesnumber
	\s^{T}\Fo \s+2\fo^T \s+\dobj \\
	\hspace{2em} = (\smsum_i \ts\hat{s}_i)^{T} \Fo (\smsum_i \ts\hat{s}_i)  +2 \fo^T (\smsum_i \ts\hat{s}_i) + \dobj 
	+\underbrace{\tv^2 v^T \Fo v }_{<0}
	+2 \tv (\smsum_i \ts \Fo \hat{s}_i + \fo)^T v  
	\IEEEnonumber \label{eq:f(x)<0-2} \\
	\s^T \Fi \s+\ei
	= \underbrace{\ts^2 \hat{s}^T_i \Fi \hat{s}_i + \ei}_{A}
	+ \underbrace{\tv^2 v^T \Fi v}_{= 0} 
	+ \underbrace{2 \ts \tv \hat{s}^T_i \Fi v}_{B}
	\IEEEyessubnumber \label{eq:g(x)=0-2}
\end{IEEEeqnarray}
The idea is the same as before: make $\tv$ large enough so that \eqref{eq:f(x)<0-2} becomes negative.
We need to verify that \eqref{eq:g(x)=0-2} equals zero for some $\ts$.
We consider two cases:
\begin{enumerate}
\item If $v\in \nullspace[\Fi]$, then $\hat{s}^T_i \Fi v=0$, so we need to choose $\hat{s}_i$ such that $A$ equals zero. Such point is guaranteed to exist because of \cref{assumption:slater}.
\item If $v\notin \nullspace[\Fi]$, then choose $\hat{s}_i$ such that $\hat{s}_i^T \Fi \hat{s}_i \neq 0$. Start with $\ts = 0$ and increase its value until $A=-B$. Choose the appropriate sign accordingly.
\end{enumerate}
We recover $\x$ with $(\x,1)^T = P \cdot (\s,1)^T$, which satisfies \eqref{eq:vector-to-equation-4}.
\end{proof}

\begin{proof}[\cref{th:2ndcase}]
Consider $Y\succeq 0$ with spectral decomposition $Y = VV^T$ where $V=[v_1,\ldots,v_R]$ and $R\leq p+1$ represents the rank of $Y$.
We have in the constraints that
\begin{equation}
	\trace[F_i Y] = \trace[V^T F_i V] = \sum_{r=1}^R \sum_{j\notin M} (F_i)_{jj}v^2_{rj} + \sum_{m\in M} 2 (F_i)_{m,p+1} v_{r,p+1} v_{r,m} = \psi_i 
	\qquad \forall i\in \N,
\end{equation} 
where $v_r = (v_{rj})_{j=1}^{p+1}$.
We can choose a rank~1 candidate solution $y=(y_j)_{j=1}^{p+1}$ such that
\begin{IEEEeqnarray}{rCl'l}
	y_j & = & (D)_{jj}\sqrt{\sum_{r}v_{ri}^{2}}, & j\notin M,  \\
	y_j & = & (D)_{p+1,p+1}\frac{\sum_{r}v_{r,p+1}v_{ri}}{y_{p+1}} = \frac{\sum_{r}v_{r,p+1}v_{ri}}{\sqrt{\sum_{r}v_{r,p+1}^{2}}}, & j\in M.
\end{IEEEeqnarray}
Using vector $y$ in the constraints gives
\begin{IEEEeqnarray}{rCl'l}
	y^T F_i y & = & \sum_{j\notin M} (F_i)_{jj} y_j^2 + \sum_{m\in M} 2 (F_i)_{m,p+1} y_{p+1} y_m \\
	          & = & \sum_{r=1}^R \sum_{j\notin M} (F_i)_{jj} v^2_{rj} + \sum_{m\in M} 2 (F_i)_{m,p+1} v_{r,p+1} v_{r,m} = \psi_i 
	& \forall i\in \N.
\end{IEEEeqnarray}
The choice of vector $y$ fulfills $\trace[F_i Y]= y^T F_i y = \psi_i$ for $i\in\N$.
We need to verify that $ y^T F_0 y \leq \trace[F_0 Y]$.
Consider the following:
\begin{IEEEeqnarray}{rCl}
	\trace[F_0 Y] & = & \trace[V^T F_0 V] = \sum_{r=1}^R v_r^T F_0 v_r 
	= \sum_{r=1}^R \sum_{i=1}^{p+1} (F_0)_{ii} v^2_{ri} +2\sum_{j>i} (F_0)_{ij} v_{ri} v_{rj}  <0  \label{eq:case2-sdp}\\
	y^T F_0 y     & = & \sum_i (F_0)_{ii} y_i^2 + 2 \sum_{j>i} (F_0)_{ij} y_i y_j  \label{eq:case2-sdp-2}
\end{IEEEeqnarray}
and compare both expressions term by term.
From \eqref{eq:case2-sdp-2} we get
\begin{IEEEeqnarray}{rCl'l}
	(F_{0})_{ii} y_i^2 & = & \sum_r (F_{0})_{ii} v_{ri}^{2} & i\notin M \label{eq:case2-y-1} \\
	(F_{0})_{ij} y_i y_j & = & \sum_r (D)_{ii}(D)_{jj} (F_{0})_{ij} \sqrt{\sum_{r} v^2_{ri}} \sqrt{\sum_{r} v^2_{rj}} & i,j\notin M  \label{eq:case2-y-2}  \\
	(F_{0})_{ii} y_i^2 & = & (F_{0})_{ii} \frac{(\sum_r v_{r,p+1}v_{ri})^2}{\sum_r v_{r,p+1}^2} & i\in M  \label{eq:case2-y-3} \\
	(F_{0})_{i,p+1} y_i y_{p+1} & = & \sum_r  (F_{0})_{i,p+1} v_{ri} v_{r,p+1} & i\in M  \label{eq:case2-y-4} \\
	(F_{0})_{i,j} & = & 0 & i\in M, j\neq p+1, j\neq i.  \label{eq:case2-y-5} 
\end{IEEEeqnarray}
Now we can establish the following relations when comparing term by term with \eqref{eq:case2-sdp}:
\begin{IEEEeqnarray}{rCl'l}
	(F_{0})_{ii} y_i^2 & = & \sum_r (F_0)_{ii} v^2_{ri} &  i\notin M \\    
	(F_{0})_{ij} y_i y_j & \leq & \sum_r (F_0)_{ij} v_{ri} v_{rj} &  i,j\notin M \text{ if } (D)_{ii}(D)_{jj}(F_0)_{ij} \leq 0 
	\IEEEeqnarraynumspace \\  
	(F_{0})_{ii} y_i^2 & \leq & \sum_r (F_0)_{ii} v^2_{ri} & i\in M \text{ if } (F_0)_{ii} \geq 0 \\ 
	(F_{0})_{i,p+1} y_i y_{p+1} & = & \sum_r  (F_{0})_{i,p+1} v_{ri} v_{r,p+1} & i\in M \\ 
	(F_{0})_{i,j} y_i y_j & = & \sum_r  (F_{0})_{i,j} v_{ri} v_{r,j} = 0 & i\in M, j\neq p+1, j\neq i. 
\end{IEEEeqnarray}
and conclude that 
\begin{IEEEeqnarray}{rCl'l}
	y^T F_0 y & \leq & \trace[F_0 Y] \\
	y^T F_i y & = & \trace[F_i Y] = \psi_i & \forall i\in\N.
\end{IEEEeqnarray}

\end{proof}

\begin{proof}[\cref{lemma:case2-back-to-sproperty}]
Implication \eqref{eq:vector-to-equation-5} $\Longleftarrow$ \eqref{eq:vector-to-equation-6} is verified, since we can introduce $z=(\x^T,1)^T=Py$:
\begin{IEEEeqnarray}{l} \IEEEyesnumber \phantomsection \label{eq:zy-domain-2case}
    f(\x)-\alpha = z^T (\Abco-\alpha \evec) z = y^T P^T (\Abco-\alpha \evec) P y = y^T F_0 y < 0 
    \IEEEyessubnumber \\
	g_i(\x) = z^T \Abci z = y^T P^T \Abci P y = y^T F_i y \unlhd 0,\quad \forall i\in\N. 
	\IEEEyessubnumber
\end{IEEEeqnarray}

The opposite direction \eqref{eq:vector-to-equation-5} $\Longrightarrow$ \eqref{eq:vector-to-equation-6} is also simple.
Consider $y=(v^T,w)^T$ with $v\in\mathbb{R}^p$ and $w\in\mathbb{R}$.
\begin{equation}
	\begin{pmatrix}
		v \\ w
	\end{pmatrix}^T
	\underbrace{\begin{pmatrix}
		\Fo & \fo \\
		\fo^T & \dobj
	\end{pmatrix}}_{F_0}
	\begin{pmatrix}
		v \\ w
	\end{pmatrix} <0,\quad
	\begin{pmatrix}
		v \\ w
	\end{pmatrix}^T
	\underbrace{\begin{pmatrix}
		\Fi & 0 \\ 
		0^T  & \ei
	\end{pmatrix}}_{F_i}
	\begin{pmatrix}
		v \\ w
	\end{pmatrix} \leq 0,\quad \forall i\in\N.
\end{equation}
If $w\neq 0$ we can choose $\s=v/w$ and \eqref{eq:vector-to-equation-6} follows.
We cannot have $w = 0$ because $\Ao\succeq 0$ by assumption, and  $v^T\Fo v\not<0$ for any $v$.
\end{proof}

\section{Proof of coercitivity (\cref{sec:existence})}
\label{appendixB}
\begin{proof}
In order to prove the coercivity of the function, we need to show that if $\Vert \blam \Vert \rightarrow \infty$, then $q(\blam) \rightarrow -\infty$.
Let's indicate  $M=\set{i\in\N\,|\,\lambda_i \rightarrow \infty}$ and $\overline{M}=\set{i\in\N\,|\,\lambda_i \nrightarrow \infty}$.
We require that $M$ is nonempty.

We distinguish two cases, one where there $\exists i\in M\,|\,\Ai \nsucceq 0$, and another where $\forall i\in M,\,\Ai \succeq 0$.
In the first case, $q(\blam) \rightarrow -\infty$ because $\Ao+\sum_i \lambda_i\Ai\nsucceq 0$ as $\lambda_i \rightarrow \infty$.
In the second case, we take $\lambda_i=\lambda$ for every $i\in M$ and we get 
\begin{IEEEeqnarray*}{l}
	\lim_{\Vert \blam \Vert\rightarrow\infty} q(\blam) =  \lim_{\Vert \blam \Vert\rightarrow\infty} 
	\smsum_i \lambda_i \cli
	-( \bo +\smsum_i \lambda_i \bi)^T  
	 (\Ao+\smsum_i \lambda_i \Ai)^\dagger( \bo + \smsum_i \bi )  \\
		= \lim_{\Vert \blam \Vert\rightarrow\infty} 
	\lambda \Big( \frac{1}{\lambda} \smsum_i \lambda_i \cli
	-\frac{1}{\lambda^2}( \bo +\smsum_i \lambda_i\bi)^T  
	 \Big(\frac{1}{\lambda}\Ao+\frac{1}{\lambda} \smsum_i \lambda_i \Ai\Big)^\dagger( \bo + \smsum_i \lambda_i\bi ) \Big) \\
		= \lim_{\Vert \blam \Vert\rightarrow\infty}\smsum_{i\in M} \lambda \Big(c_i - \bi^T\Big(\frac{1}{\lambda}\Ao+\frac{1}{\lambda}\smsum_{j\in\overline{M}} \lambda_j \widetilde{A}_j +\smsum_{j\in M}\widetilde{A}_j \Big)^{\dagger} \bi \Big) \\
		= \lim_{\Vert \blam \Vert\rightarrow\infty} \smsum_{i\in M} \lambda (\cli -\bi^T  \Ai^\dagger \bi). \IEEEyesnumber
\label{eq:limit-q}
\end{IEEEeqnarray*}
Note that we used the relations
\begin{equation}
	\Big(\smsum_{i\in M} \Ai \Big)^{\dagger} = \smsum_{i\in M} \Ai^{\dagger}\quad\text{and}\quad
	\bi \Ai^{\dagger} b_j = {b}_j \Ai^{\dagger} {b}_j=0\quad\forall i\neq j
\end{equation}
because of the non-overlapping property of the $\Ai$'s and $\bi$'s.

Considering the case where $\Ali\succeq 0$, we can express $g_i(x_i)$ in the following form:
\begin{equation}
	g_i(x_i)  
		= \big\Vert (\Ali)^{1/2}x_i +\bli^{\text{sq}}   \big\Vert^2 - c_i^{sq} 
	 = x_i^T \Ali x_i+2\bli^T x_i+\cli,
\end{equation}
where $\Ali= (\Ali)^{T/2}(\Ali)^{1/2}$, $\bli^T = (\bli^{\text{sq}})^T (\Ali)^{1/2}$, $\cli = (\bli^{\text{sq}})^T\bli^{\text{sq}}- c_i^{sq}$.
Because of Slater's condition there exists an $\hat{x}_i$ such that $g_i(\hat{x}_i)<0$ and we necessarily have $c_i^{sq}>0$. 
Determining the value of \eqref{eq:limit-q} we get:
\begin{equation*}
	\cli - \bli^T \Ali^{\dagger}\bli = (\bli^{\text{sq}})^T\bli^{\text{sq}}- c_1^{sq}
	-(\bli^{\text{sq}})^T (\Ali)^{1/2}(\Ali)^{\dagger/2}(\Ali)^{\dagger T/2}(\Ali)^{T/2}(\bli^{\text{sq}}) 
	= - c_1^{sq} < 0
\end{equation*}
and conclude that $\lim_{\Vert \blam \Vert\rightarrow\infty} q(\blam) = -\infty$.
The case in which $\lambda_i \rightarrow -\infty$ is analyzed in a similar way and we omit it.
\end{proof}

\section{Proof of strong duality in the robust least squares problem (\cref{sec:rls})}
\label{appendixC}
\begin{proof}[\cref{th:rls}]
We want to show that problem \eqref{eq:rls-qp2} presents the strong duality property:
\begin{equation}
\begin{IEEEeqnarraybox}[][c]{r'l'l}
	\max_{\Delta} & \Vert (H + \Delta) \overline{x} \Vert^2 \\
	\text{s.t.} & \Vert \Delta_{:i} \Vert^2 \leq \rho_i & \forall i\in \set{1,\ldots,p+1}.
\end{IEEEeqnarraybox}
\label{eq:rls-qp2}
\end{equation}
We can rewrite the objective function of problem \eqref{eq:rls-qp2} in quadratic form as follows:
\begin{IEEEeqnarray}{rCl}
	\Vert (H + \Delta) \overline{x} \Vert^2 & = & \overline{x}^T (H+\Delta)^T (H+\Delta) \overline{x}
	= \Big( \sum_i \overline{x}_i (H_{:i}+\Delta_{:i})^T \Big) 
	\Big( \sum_i \overline{x}_i (H_{:i}+\Delta_{:i}) \Big) 
	\IEEEnonumber \\
	& = & \sum_i \overline{x}_i^2 (H_{:i}+\Delta_{:i})^T (H_{:i}+\Delta_{:i}) + \sum_{j\neq i} 2 \overline{x}_i \overline{x}_j (H_{:j}+\Delta_{:j})^T (H_{:i}+\Delta_{:i}),
\label{eq:rls-qp-in-proof}
\end{IEEEeqnarray}
where $\overline{x}_i$ refers to the $i$th component of vector $\overline{x}$ (a scalar), and $\Delta_{:i}$ or $H_{:i}$ refer to the $i$th column of matrix $\Delta$ or $H$, respectively.
We now reformulate \eqref{eq:rls-qp-in-proof} to standard quadratic form:
\begin{equation}
	\Vert (H + \Delta) \overline{x} \Vert^2 =\col [\Delta + H]^T 
	( \overline{x}\overline{x}^T \otimes \id_{N})  \col [\Delta+ H],
\end{equation}
where $\col [\Delta] = (\Delta^T_{:1},\ldots,\Delta^T_{:p+1})^T$ and $\otimes$ represents the Kronecker product.
We now make the following change of variable $\Delta' = D' \col [\Delta+ H]$, where $D'$ is a diagonal matrix derived as follows: $D' = \diag[\sign[\overline x]]\otimes \id_{N}$, where $\sign[\overline x]$ returns a vector with $\pm 1$ elements depending on the sign of $\overline{x}$, and $\sign[0] = +1$.
We get the following
\begin{IEEEeqnarray}{rCl}
	\Vert (H + \Delta) \overline{x} \Vert^2 & = & \col [\Delta + H]^T D'D'( \overline{x}\overline{x}^T \otimes \id_{N})D'D'  \col [\Delta+ H] \IEEEnonumber \\
	& = & \col [\Delta + H]^T D'( |\overline{x}\overline{x}^T| \otimes \id_{N})D'  \col [\Delta+ H]
	= (\Delta')^T \Ao \Delta',
\end{IEEEeqnarray}
where $\Ao = (|\overline{x}\overline{x}^T|\otimes \id_N)$ is a completely positive matrix.
Recall that matrix $A$ is completely positive if and only if there exists a decomposition such that $A=BB^T$, where $B$ has non-negative elements.
We can now reformulate problem \eqref{eq:rls-qp2} into:
\begin{equation}
\begin{IEEEeqnarraybox}[][c]{r'l'l}
	\max_{\Delta} & (\Delta')^T \Ao \Delta' \\
	\text{s.t.} & \Vert D'_i \Delta'_{:i}-H_{:i} \Vert^2 \leq \rho_i & \forall i\in \set{1,\ldots,p+1},
\end{IEEEeqnarraybox}
\label{eq:rls-qp3}
\end{equation}
where $D'_i$ is obtained from $D'$.
Problem \eqref{eq:rls-qp3} is in standard form, as introduced in \cref{sec:main}.
We define matrix $P$ as presented in \eqref{eq:p-matrix}, where $P_i = \id_{N}$ and $p_i = D_i' H_i$ in this particular case.
This choice of matrix $P$ guarantees that $F_i = P^T \Abci P$ are all diagonal for $i\in\{1,\ldots,p+1\}$.

We need to verify that there exists some diagonal matrix $D''$, such that $D'' F_0 D''$ is a completely positive matrix, where $F_0=P^T \Ao P$.
If this matrix $D''$ exists and since we are maximazing in~\eqref{eq:rls-qp3}, $-F_0$ becomes a Z-matrix and the requirements of \cref{theorem:sproperty-result} are fulfilled, ultimately guaranteeing the \sproperty\ and the strong duality result we are interested in.

We know that $\Ao$ is a completely positive matrix. 
Considering that $\Ao$ can be easily decomposed into $N$ blocks and each block is also completely positive with the following form:
\begin{equation}
	(\Ao)_i =
	\begin{pmatrix}
		|\overline{x}| \\ 0 
	\end{pmatrix}
	\begin{pmatrix}
		|\overline{x}^T| & 0
	\end{pmatrix},
\end{equation}
we can compute $(P^T)_i (\Ao)_i (P)_i$ and verify if there exists submatrix $(D'')_i $ such that the after product is completely positive:
\begin{IEEEeqnarray}{rCl}
	(D'')_i (P^T)_i (\Ao)_i (P)_i (D'')_i  & = &
	\begin{pmatrix}
		D_i'' & 0 \\
		0 & 1
	\end{pmatrix}
	\begin{pmatrix}
		I_{p+1} & 0 \\
		H_{i:} & 1
	\end{pmatrix}
	\begin{pmatrix}
		|\overline{x}| \\ 0 
	\end{pmatrix}
	\begin{pmatrix}
		|\overline{x}^T| & 0
	\end{pmatrix}
	\begin{pmatrix}
		I_{p+1} & H_{i:}^T \\
		0 & 1
	\end{pmatrix}
	\begin{pmatrix}
		D_i'' & 0 \\
		0 & 1
	\end{pmatrix} \IEEEnonumber \\
	& = &
	\begin{pmatrix}
		D_i'' & 0 \\
		0 & 1
	\end{pmatrix}
	\begin{pmatrix}
		|\overline{x}| \\
		H_{:i} |\overline{x}|
	\end{pmatrix}
	\begin{pmatrix}
		|\overline{x}^T| & |\overline{x}^T|  H_{:i}^T
	\end{pmatrix}
	\begin{pmatrix}
		D_i'' & 0 \\
		0 & 1
	\end{pmatrix} \IEEEnonumber \\
	& = &
	\begin{pmatrix}
		D_i''|\overline{x}| \\
		H_{:i} |\overline{x}|
	\end{pmatrix}	
	\begin{pmatrix}
		|\overline{x}|^T D_i'' & |\overline{x}|^T H_{:i} 
	\end{pmatrix}
\end{IEEEeqnarray}
where $H_{:i}$ refers to the row vector of matrix $H$ and $(A)_i$ refers to the corresponding submatrix of adecuate dimensions.

If $H_{:i} |\overline{x}|$ is positive or zero, choose $D_i''=\id_{N}$ and the matrix $(D'')_i  (P^T)_i (\Ao)_i (P)_i (D'')_i $ becomes completely positive.
If $H_{:i} |\overline{x}|$ is negative, choose $D_i''=-\id_{N}$, and, likewise, matrix $(D'')_i  (P^T)_i (\Ao)_i (P)_i (D'')_i $ becomes completely positive too. 

The previous analysis can be extended to every other block matrix and $H_{i:}$.
This guarantees that we can always find a $D''$ matrix such that $D'' F_0 D''$ is a completely positive matrix, which concludes the proof.
\end{proof}



\bibliographystyle{spmpsci}      
\bibliography{quadratic}   

\end{document}